\documentclass[12pt]{amsart}
\usepackage[shortlabels]{enumitem}
\usepackage{SSdefn}

\newcommand{\lf}{\mathrm{lf}}

\newcommand{\gen}{{\rm gen}}

\newcommand{\fgen}{{\rm fg}}
\newcommand{\FI}{\mathbf{FI}}
\newcommand{\OI}{\mathbf{OI}}

\author{Steven V Sam}
\address{Department of Mathematics, University of California, San Diego, CA}
\email{\href{mailto:ssam@ucsd.edu}{ssam@ucsd.edu}}
\urladdr{\url{http://math.ucsd.edu/~ssam/}}
\thanks{SS was supported by NSF grant DMS-1849173.}

\author{Andrew Snowden}
\address{Department of Mathematics, University of Michigan, Ann Arbor, MI}
\email{\href{mailto:asnowden@umich.edu}{asnowden@umich.edu}}
\urladdr{\url{http://www-personal.umich.edu/~asnowden/}}
\thanks{AS was supported by NSF grant DMS-1453893.}

\title[Sp-equivariant modules]{Sp-equivariant modules over polynomial rings\\ in infinitely many variables}

\date{June 11, 2021}
\subjclass[2010]{%
  13E05, %   	Noetherian rings and modules
  13A50%   	Actions of groups on commutative rings; invariant theory
}

\begin{document}

\begin{abstract}
We study the category of $\Sp$-equivariant modules over the infinite variable polynomial ring, where $\Sp$ denotes the infinite symplectic group. We establish a number of results about this category: for instance, we show that every finitely generated module $M$ fits into an exact triangle $T \to M \to F \to$ where $T$ is a finite length complex of torsion modules and $F$ is a finite length complex of ``free'' modules; we determine the Grothendieck group; and we (partially) determine the structure of injective modules. We apply these results to show that the twisted commutative algebras $\Sym(\bC^{\infty} \oplus \lw^2{\bC^{\infty}})$ and $\Sym(\bC^{\infty} \oplus \Sym^2{\bC^{\infty}})$ are noetherian, which are the strongest results to date of this kind. We also show that the free 2-step nilpotent twisted Lie algebra and Lie superalgebra are noetherian.
\end{abstract}

\maketitle
\tableofcontents

\section{Introduction}

In \cite{brauercat}, we study the representation theory of several categories modeled on Brauer algebras and their variants. We discover that this theory is deeply connected to several other branches of representation theory, including supergroups, parabolic category $\cO$, and twisted commutative algebras. While the first two have been studied extensively, twisted commutative algebras have only recently begun to be studied because of their importance in representation stability. Several of the examples relevant to \cite{brauercat} were considered in our previous work \cite{sym2noeth, periplectic}. However, one was not, and the starting point of this paper is to analyze it. (Actually, it is the twisted Lie algebra $\fg$ discussed below that is most relevant.)

\subsection{Results on tca's}

Let $\Rep^{\pol}(\GL)$ be the category of polynomial representations of the infinite linear group $\GL$ (over the complex numbers), and let $\bV$ be the standard representation of $\GL$. For the purposes of this paper, a {\bf twisted commutative algebra} (tca) is an algebra object in this category. A major open problem in tca theory is determining if finitely generated tca's are noetherian. This is shown to be the case for so-called bounded tca's in \cite{delta-mod}, and for the tca's $\Sym(\lw^2{\bV})$ and $\Sym(\Sym^2{\bV})$ in \cite{sym2noeth} (see also \cite{periplectic} for some related results). Before this paper, these were the only known cases. We add two more:

\begin{theorem} \label{thm:tca}
The tca's $\Sym(\bV \oplus \lw^2{\bV})$ and $\Sym(\bV \oplus \Sym^2{\bV})$ are noetherian.
\end{theorem}

We also establish some properties of these tca's that are known in other cases, e.g., projective modules are injective, and the generic category is equivalent to the category of modules supported at~0. We suspect our methods would allow one to prove that the tca's $\Sym(\bV^{\oplus n} \oplus \lw^2{\bV})$ and $\Sym(\bV^{\oplus n} \oplus \Sym^2{\bV})$ are noetherian for any $n$, though we have not pursued this.

Let $\fg$ be the free 2-step nilpotent Lie algebra $\bV \oplus \lw^2{\bV}$ in the category $\Rep^{\pol}(\GL)$. The following result is an easy consequence of Theorem~\ref{thm:tca} and was our original motivation:

\begin{theorem}
The category of $\fg$-modules in $\Rep^{\pol}(\GL)$ is locally noetherian.
\end{theorem}

This theorem can be equivalently stated as: the module category for the upwards spin-Brauer category (as defined in \cite{spincat}) is locally noetherian. This will be used in \cite{brauercat}.

\subsection{$\Sp$-equivariant modules}

We showed in \cite{sym2noeth} that $\GL$-equivariant modules over $\Frac(\Sym(\lw^2{\bV}))$ (with a polynomiality condition) are equivalent to algebraic representations of the infinite symplectic group $\Sp$ (see \S \ref{ss:repSp} for the definition of algebraic). Thus if $M$ is a module over the tca $\Sym(\bV \oplus \lw^2{\bV})$ then we can tensor it up to $\Frac(\Sym(\lw^2{\bV}))$ and pass through this equivalence to obtain an $\Sp$-equivariant module over the algebra $\Sym(\bV)$. Due to this, most of the work in proving Theorem~\ref{thm:tca} involves studying such modules. We now outline our results in this direction.

Let $\Rep(\Sp)$ denote the category of algebraic representations of $\Sp$ and let $A=\Sym(\bV)$, regarded as an algebra object in $\Rep(\Sp)$. Let $\Mod_A$ be the category of $A$-modules in the category $\Rep(\Sp)$. An $A$-module is called {\bf torsion} if every element has non-zero annihilator. The category $\Mod_A^{\tors}$ of torsion modules is a Serre subcategory of $\Mod_A$, and so we can consider the quotient $\Mod_A^{\gen}$, which we call the {\bf generic category}. Our strategy for understanding $\Mod_A$ is to first understand $\Mod_A^{\gen}$. This idea is motivated by our previous work \cite{symc1}, where we study $\GL$-equivariant modules over $\Sym(\bV)$.

Consider $\Sp_{2n}$ acting on the polynomial ring in $2n$ variables. Equivariant modules correspond to equivariant quasi-coherent sheaves on $\bA^{2n}$, and torsion modules correspond to sheaves supported at~0. The Serre quotient category is therefore equivalent to the category of $\Sp_{2n}$-equivariant sheaves on $\bA^{2n} \setminus \{0\}$. Since $\Sp_{2n}$ acts transitively on this space, such sheaves correspond to representations of any stabilizer group.

Taking this picture as our guide, we let $\xi \in \Spec(A)=\bV^*$ be the functional on $\bV$ defined by $\xi(e_i)=\xi(f_i)=1$ for $i \ge 1$ (where $\{e_i, f_i\}_{i \ge 1}$ is our standard symplectic basis of $\bV$), and we let $H \subset \Sp$ be its stabilizer. We define a category $\Rep(H)$ of algebraic representations of $H$, and prove:

\begin{theorem} \label{thm:gen0}
We have a natural equivalence of categories $\Mod_A^{\gen} \cong \Rep(H)$.
\end{theorem}

While we have just explained why this theorem is, in a sense, intuitively obvious, it is technically demanding to prove, and there is an important subtlety that the heuristic does not capture: the theorem is false for certain other choices of $\xi$. See \S \ref{ss:xi} for more.

In \cite{infrank}, we studied algebraic representations of infinite rank classical groups. Using the methods of that paper, we analyze algebraic representations of $H$, and prove a number of results:
\begin{itemize}
\item The representations $\bS_{\lambda}(\bV)$ are injective in $\Rep(H)$. Here $\bS_{\lambda}$ denotes a Schur functor.
\item The simple objects of $\Rep(H)$ can be obtained by a variant of Weyl's construction.
\item Every finite length object of $\Rep(H)$ has finite injective dimension.
\item $\Rep(H)$ is equivalent to the category of modules over the twisted commutative algebra $\Sym(\bV \oplus \lw^2{\bV})$ supported at~0.
\item While we do not actually carry out the details to prove this, our results imply that $\Rep(H)$ is the universal $\bC$-linear tensor category equipped with an object $\bV$ admitting a pairing $\lw^2{\bV} \to \bC$ and a functional $\bV \to \bC$. (See \cite[4.4.2]{infrank} for a similar result.)
\end{itemize}
Via Theorem~\ref{thm:gen0}, we can transfer all of these results back to $\Mod_A^{\gen}$.

With these results about $\Mod_A^{\gen}$ in hand, we turn our attention back to the main category of interest, $\Mod_A$. We prove a number of results, such as:
\begin{itemize}
\item If $M$ is a finitely generated $A$-module then there is an exact triangle $T \to M \to F \to$, where $T$ is a finite length complex of finite length $A$-modules and $F$ is a finite length complex of $A$-modules of the form $V \otimes A$, where $V$ is a polynomial representation of $\GL$ (restricted to $\Sp$). This is somewhat analogous to the structure theorem for modules over a PID, though only works at the derived level. It is a very powerful structural result for $A$-modules: indeed, the analog in the $\GL$-case (or for $\FI$-modules) has proven to be one of the most important tools in that theory.
\item We determine the Grothendieck group of the category $\Mod_A$: it is naturally a module over the ring of symmetric functions, and, as such, free of rank two. The classes $[A]$ and $[\bC]$ form a basis. (Here $\bC$ is an $A/A_+$-module.)
\item We show that every injective $A$-module $I$ decomposes as $I' \oplus I''$, where $I'$ is a torsion injective module and $I''$ is a torsion-free injective $A$-module. We show that the torsion-free injective $A$-modules are exactly the modules of the form $V \otimes A$ where $V$ is a polynomial representation of $\GL$. Thus, in a sense, free $A$-modules are injective. We have not been able to determine the structure of torsion injectives; however, we show by example that the indecomposable ones need not be finite length. (We also show that finite length $A$-modules need not have finite injective dimension.)
\item We define a version of local cohomology for $A$-modules at the maximal ideal $A_+$. We show that if $M$ is a finitely generated $A$-module then all of its local cohomology groups have finite length, and only finitely many of them are non-zero.
\end{itemize}
These results are sufficient to allow us to prove Theorem~\ref{thm:tca}, using a method similar to the one employed in \cite{sym2noeth} to prove noetherianity of $\Sym(\lw^2{\bV})$.

\begin{remark}
While our interest in $\Sp$-equivariant $\Sym(\bV)$-modules was to prove Theorem~\ref{thm:tca}, they may well turn out to be of interest. Indeed, they are closely related to $\GL$-equivariant modules over $\Sym(\bV)$, and these are equivalent to $\FI$-modules \cite{fimodule}, which have seen numerous applications.
\end{remark}

\subsection{Open problems}

Here are a few interesting problems that we have not addressed:
\begin{itemize}
\item Compute the derived section functor on simple objects of $\Mod_A^{\gen}$. The analogous result in the $\GL$-case appears in \cite[\S 7.4]{symc1}.
\item Describe the indecomposable injective objects in $\Mod_A^{\tors}$.
\item Describe the structure of injective resolutions in $\Mod_A^{\tors}$.
\item Compute the $\Ext$ groups between simple $A$-modules. In the one computation we have done (Example~\ref{ex:trivial-ext}), the result is 2-periodic. Does this happen more generally?
\end{itemize}

\subsection{Outline}

In \S \ref{s:bg}, we review relevant background information and prove some simple preliminary results. In \S \ref{s:repH}, we study the representation category of the stabilizer group $H$. In \S \ref{s:loc}, we study the local structure of $A$-modules at the point $\xi \in \Spec(A)$. In \S \ref{s:gen}, we study the generic category $\Mod_A^{\gen}$ and show that it is equivalent to $\Rep(H)$. In \S \ref{s:Amod}, we apply our work on the generic category to deduce results about $A$-modules. In \S \ref{s:Bnoeth}, we prove that the tca $\Sym(\bV \oplus \lw^2{\bV})$ is noetherian. Finally, in \S \ref{s:further} we discuss some additional results.

\subsection{Notation}

We list some of the important notation:
\begin{description}[align=right,labelwidth=2cm,leftmargin=!]
\item [$\cA^{\fin}$] the category of finite length objects in the abelian category $\cA$
\item [$\cA^{\lf}$] the category of locally finite length objects in the abelian category $\cA$
\item [$\bV$] the complex vector space with basis $\{e_i,f_i\}_{i \ge 1}$
\item [$\GL$] the group of automorphisms of $\bV$ fixing all but finitely many basis vectors
\item [$\bS_{\lambda}$] the Schur functor associated to the partition $\lambda$
\item [$\bV_{\lambda}$] the space $\bS_{\lambda}(\bV)$, considered as a representation of $\GL$ or any subgroup
\item [$\omega$] the symplectic form on $\bV$ with $\omega(e_i, f_i)=1$
\item [$\xi$] the linear functional on $\bV$ defined by $\xi(e_i)=\xi(f_i)=1$ for all $i$
\item [$\Sp$] the subgroup of $\GL$ fixing $\omega$
\item [$H$] the subgroup of $\Sp$ fixing $\xi$
\item [$A$] the polynomial ring $\Sym(\bV)$, regarded with its $\Sp$-action
\item [$B$] the tca $\Sym(\bV \oplus \lw^2{\bV})$
\end{description}

\section{Background} \label{s:bg}

\subsection{Polynomial representations of $\GL$}

Let $\bV$ be a complex vector space of countably infinite dimension. We let $\{e_i, f_i\}_{i \ge 1}$ be a basis for $\bV$. We let $\GL$ be the group of automorphisms of $\bV$ that fix all but finitely many basis vectors. The space $\bV$ is tautologically a representation of $\GL$. We say that a representation of $\GL$ is {\bf polynomial} if it can be realized as a subquotient of a (possibly infinite) direct sum of tensor powers of $\bV$. We let $\Rep^{\pol}(\GL)$ denote the category of polynomial representations of $\GL$. It is a semi-simple abelian category, and closed under tensor product. For a partition $\lambda$, we let $\bV_{\lambda}$ be the polynomial representation $\bS_{\lambda}(\bV)$, which is known to be irreducible. Every irreducible polynomial representation is isomorphic to $\bV_{\lambda}$ for a unique $\lambda$. For additional information on polynomial representations, see \cite{expos}.

\subsection{Twisted commutative algebras}

A {\bf twisted commutative algebra} (tca) is (for us) an algebra object in the tensor category $\Rep^{\pol}(\GL)$. Thus a tca is an ordinary commutative $\bC$-algebra $R$ equipped with an action of $\GL$ under which it forms a polynomial representation. By a {\bf module} over a tca $R$ we mean a module object in $\Rep^{\pol}(\GL)$. Thus a module is an ordinary $R$-module $M$ equipped with a compatible action of $\GL$ under which it forms a polynomial representation. When $R$ is a tca, ``$R$-module'' will by default be taken in the sense of tca's; we use the notation ``$\vert R \vert$-module'' to refer to non-equivariant modules, in the few cases there is such a need. An $R$-module is finitely generated if it contains finitely many elements whose $\GL$-orbits generate it as an $\vert R \vert$-module. If $V$ is a polynomial representation of $\GL$ then $V \otimes R$ is a projective $R$-module, and every $R$-module is a quotient of one of this form. An $R$-module is finitely generated if and only if it can be realized as a quotient of $V \otimes R$ for some finite length polynomial representation $V$. We say that a tca $R$ is {\bf noetherian} if every submodule of a finitely generated module is again finitely generated. For additional information on tca's, see \cite{expos}.

\subsection{Comparison of projectives and injectives} \label{ss:inj-proj}

The category $\Rep^{\pol}(\GL)$ has an internal $\Hom$, which we denote by $\uHom$. It can be defined explicitly by the formula
\begin{displaymath}
\uHom(V,W) = \bigoplus\nolimits \Hom(V \otimes \bV_{\lambda}, W) \otimes \bV_{\lambda},
\end{displaymath}
where the sum is over all partitions $\lambda$. There is a canonical isomorphism
\begin{displaymath}
\Hom(U, \uHom(V,W)) \cong \Hom(U \otimes V, W).
\end{displaymath}
Let $R$ be a tca with $R_0=\bC$ and let $V$ be a polynomial representation. Put $I_R(V)=\uHom(R, V)$. Then $I_R(V)$ is naturally an $R$-module. Moreover, for any $R$-module $M$, we have a natural isomorphism
\begin{displaymath}
\Hom_R(M, I_R(V)) \cong \Hom_{\GL}(M, V).
\end{displaymath}
We thus see that $I_R(V)$ is an injective $R$-module; in fact, it is easily seen to be the injective envelope of $V$, regarded as an $R$-module with $R_+$ acting by~0. Let $P_R(V)=R \otimes V$, which is a projective $R$-module and has a dual mapping property to $I_R(V)$. We will require the following relationship between $P_R(V)$ and $I_R(V)$:

\begin{proposition} \label{prop:inj-proj}
For polynomial representations $V$ and $W$, we have
\begin{displaymath}
\dim \Hom_R(P_R(V), P_R(W)) = \dim \Hom_R(I_R(V), I_R(W))
\end{displaymath}
\end{proposition}

\begin{proof}
Since $\Rep^{\pol}(\GL)$ is semi-simple, we have $\dim \Hom_{\GL}(V, W)=\dim \Hom_{\GL}(W, V)$ for any $V,W \in \Rep^{\pol}(\GL)$. We use this repeatedly in the following derivation:
\begin{align*}
\dim \Hom_R(I_R(V), I_R(W))
&= \dim \Hom_{\GL}(I_R(V), W) \\
&= \dim \Hom_{\GL}(W, \uHom(R, V)) \\
&= \dim \Hom_{\GL}(V, R \otimes W) \\
&= \dim \Hom_R(P_R(V), P_R(W)).
\end{align*}
In the first step, we used the mapping property for $I_R(W)$; in the second, the definition of $I_R(V)$; in the third, the adjunction for $\uHom$; in the final step, the mapping property for $P_R(V)$.
\end{proof}

\subsection{Algebraic representations of $\Sp$} \label{ss:repSp}

Let $\omega \colon \bV \times \bV \to \bC$ be the alternating bilinear form given by
\begin{displaymath}
\omega(e_i, f_j)=\delta_{i,j}, \qquad \omega(e_i, e_j)=0, \qquad \omega(f_i, f_j)=0,
\end{displaymath}
for $i,j \ge 1$. We let $\Sp \subset \GL$ be the subgroup preserving $\omega$. It is (one version of) the infinite symplectic group.

Since $\Sp$ is a subgroup of $\GL$, any representation of $\GL$ can be restricted to one of $\Sp$. We say that a representation of $\Sp$ is {\bf algebraic} if it occurs as a subquotient of a restriction of a polynomial representation. We let $\Rep(\Sp)$ denote the category of algebraic representations. This category was studied in detail in \cite{infrank} (see also \cite{koszulcategory, olshanskii, penkovserganova, penkovstyrkas}). We now recall the salient features of this category:
\begin{enumerate}
\item \label{sp:ss} Algebraic representations of $\Sp$ are \emph{not} semi-simple in general: for example, the map $\omega \colon \lw^2{\bV} \to \bC$ is a non-split surjection.
\item \label{sp:finlen} Every object of $\Rep(\Sp)$ is locally of finite length, i.e., the union of its finite length subobjects. Moreover, $\bV^{\otimes n}$ has finite length for all $n$.
\item \label{sp:inj} If $V$ is a polynomial representation of $\GL$ then its restriction to $\Sp$ is injective in $\Rep(\Sp)$, and every injective of $\Rep(\Sp)$ is obtained in this manner \cite[4.2.9]{infrank}. Moreover, every finite length object of $\Rep(\Sp)$ has finite injective dimension \cite[4.3.5]{infrank}.
\item \label{sp:poly-quo} Every object of $\Rep(\Sp)$ is a quotient of the restriction of some polynomial representation of $\GL$. This result does not appear in \cite{infrank}, but follows easily from \cite[Theorem~3.1]{sym2noeth}, as we now explain. We momentarily let $A$ denote the tca $\Sym(\lw^2{\bV})$ as in loc.\ cit. The form $\omega$ induces an $\Sp$-equivariant algebra homomorphism $\phi \colon A \to \bC$. The cited theorem (or, more accurately, its $\lw^2$ variant; see \cite[\S 3.5]{sym2noeth}) implies that every algebraic representation of $\Sp$ has the form $M \otimes_{A,\phi} \bC$ for some $A$-module $M$. Writing $M$ as a quotient of $V \otimes A$ for some polynomial representation $V$, we find that $M \otimes_{A,\phi} \bC$ is a quotient of $V$, as required.
  
\item \label{sp:simples} The simple objects of $\Rep(\Sp)$ are indexed by partitions, and obtained by Weyl's construction. We now recall what this means. Suppose that $V$ is a vector space equipped with an alternating form $\lw^2{V} \to \bC$. For $1 \le i<j \le n$, we have a map $t_{i,j} \colon V^{\otimes n} \to V^{\otimes (n-2)}$ obtained by applying the form to the $i$th and $j$th tensor factors. We let $V^{[n]}$ be the intersection of the kernels of the $t_{i,j}$'s; this is called the space of {\bf traceless tensors}. The symmetric group $\fS_n$ acts on $V^{[n]}$. Letting $\cM_{\lambda}$ denote the Specht module associated to $\lambda$, we define $\bS_{[\lambda]}(V)$ to be the space $\Hom_{\fS_n}(\cM_{\lambda}, V^{[n]})$, which is a representation of $\Sp(V)$. We can in particular apply this construction to our space $\bV$ with the form $\omega$. The resulting objects $\bS_{[\lambda]}(\bV)$ are simple, mutually non-isomorphic, and exhaust the simple objects of $\Rep(\Sp)$ \cite[4.1.4]{infrank}. Furthermore, every other simple object $\bS_{[\mu]}(\bV)$ appearing in $\bV_\lambda$ satisfies $|\mu|<|\lambda|$ \cite[Proposition 4.1.9]{infrank}.
  
\item \label{sp:len} For a partition $\lambda$, let $\ell(\lambda)$ denote the number of rows in $\lambda$. For a representation $V$ of $\Sp$, let $\ell(V)$ denote the supremum of $\ell(\lambda)$ taken over partitions $\lambda$ for which $\bS_{[\lambda]}{\bV}$ occurs as a constituent of $V$. Then $\ell(V \otimes W) \le \ell(V) + \ell(W)$. This follows from \cite[7.5, Theorem 4.3.4]{infrank}.

\item \label{sp:tca} The category $\Rep(\Sp)$ is equivalent to the category of locally finite length modules over the tca $R=\Sym(\lw^2{\bV})$ \cite[4.3.2]{infrank}. Under this equivalence, the irreducible representation $\bS_{[\lambda]}{\bV}$ of $\Rep(\Sp)$ corresponds to the simple $R$-module $\bV_{\lambda}$ (with $R_+$ acting by~0), and the injective representation $\bV_{\lambda}$ corresponds to the injective $R$-module $I_R(\bV_{\lambda})$ defined in \S \ref{ss:inj-proj}. Note that this implies that $\bV_\lambda$ is the injective envelope of $\bS_{[\lambda]} \bV$.

\item \label{sp:sp} Let $\bV_{\le n}$ be the span of the vectors $e_i$ with $\vert i \vert \le n$, and let $\bV_{\ge n}$ be defined analogously. Let $G_n$ be the symplectic group on $\bV_{\ge n}$; we write $\Sp_{2n}$ for the symplectic group on $\bV_{\le n}$. Then $\Sp_{2n} \times G_n$ is naturally a subgroup of $\Sp$. Given an algebraic representation $V$ of $\Sp$, it follows that $\Gamma_n(V)=V^{G_n}$ is a representation of $\Sp_{2n}$, and this defines a functor
\begin{displaymath}
\Gamma_n \colon \Rep(\Sp) \to \Rep(\Sp_{2n})
\end{displaymath}
called the {\bf specialization functor}. It is obviously left-exact. We show \cite[\S 4.4]{infrank} that it is a strict tensor functor, i.e., the natural map $\Gamma_n(V) \otimes \Gamma_n(W) \to \Gamma_n(V \otimes W)$ is an isomorphism. For any finite length representation $V$ of $\Sp$, $\Gamma_n(V)$ is finite length \cite[4.4.4]{infrank}. Furthermore, $\rR^i \Gamma_n(V)=0$ for $n$ fixed and $i$ large, or for $i>0$ and $n$ large (in fact, $n \ge \ell(V)$ suffices); this follows from \cite[4.4.6]{infrank} and the computations in \cite{littlewood}. 
\end{enumerate}

\subsection{The infinite symmetric group} \label{ss:repSym}

Let $\fS=\bigcup_{n \ge 1} \fS_n$ be the infinite symmetric group. Let $\bW \subset \bV$ be the span of the $e_i$'s with $i \ge 1$. Then $W$ is naturally a representation of $\fS$ via $\sigma e_i = e_{\sigma(i)}$. We say that a representation of $\fS$ is {\bf algebraic} if it occurs as a subquotient of a direct sum of tensor powers of $\bW$. We let $\Rep(\fS)$ denote the category of algebraic representations. This category was studied in detail in \cite[\S 6]{infrank}. We will need to use a few properties of it:
\begin{enumerate}
\item \label{sym:fin-len} Every object of $\Rep(\fS)$ is locally finite length, and $\bW^{\otimes n}$ has finite length for all $n$.
\item \label{sym:sp-res} Define a representation of $\fS$ on $\bV$ by $\sigma(e_i)=e_{\sigma(i)}$ and $\sigma(f_i)=f_{\sigma(i)}$. Then $\bV \cong \bW^{\oplus 2}$. Since $\fS$ preserves the symplectic form $\omega$ on $\bV$, it follows that we get an embedding $\fS \to \Sp$. We can therefore restrict representations of $\Sp$ to $\fS$. It is clear that algebraic representations restrict to algebraic representations, so we have a functor $\Rep(\Sp) \to \Rep(\fS)$. Note that if $V$ is a finite length algebraic representation of $\Sp$ then its restriction to $\fS$ is also of finite length. Indeed, it suffices to verify this for $V=\bV^{\otimes n}$, and the restriction of this is $(\bW^{\oplus 2})^{\otimes n}$, which has finite length.
\item \label{sym:large-n} Let $\fS_{>n}$ be the subgroup fixing each of $1, \ldots, n$. Note that $\fS_n \times \fS_{>n}$ is naturally a subgroup of $\fS$. If $V$ is a finite length algebraic representation of $\fS$ then $V^{\fS_{>n}}$ generates $V^{\fS_{>n+1}}$ as an $\fS_{n+1}$-module for $n \gg 0$. This follows from \cite[6.2.3]{infrank}, which shows that the sequence $V^{\fS_{>n}}$ can be given the structure of a finitely generated $\FI$-module where $\FI$ is the category of finite sets and injective functions.
\item \label{sym:noeth} Let $R_n=\Sym(\bW^{\oplus n})$, regarded as an algebra object in $\Rep(\fS)$. Then $R_n$ is noetherian, that is, if $M$ is a finitely generated $R_n$-module in $\Rep(\fS)$ then any submodule of $M$ is also finitely generated. This was proved for ideals of $R_n$ by Cohen \cite{cohen, cohen2} (see also \cite{AschenbrennerHillar, HillarSullivant}). The proof can easily be adapted to handle the module case; alternatively, one can appeal to \cite[ Theorem 4.6, Corollary 6.16]{nagel-romer}.
\end{enumerate}

\subsection{The algebra $A$}

Let $A$ be the algebra object $\Sym(\bV)$ in the category $\Rep(\Sp)$. We identify $A$ with the polynomial ring in variables $\{x_i, y_i\}_{i \ge 1}$, with $x_i$ corresponding to $e_i$ and $y_i$ to $f_i$. By an {\bf $A$-module}, we will always mean a module object in $\Rep(\Sp)$. We let $\Mod_A$ denote the category of $A$-modules. As with tca's, we use the term ``$\vert A \vert$-module'' to refer to non-equivariant modules, when needed. We now establish some basic properties of $A$ and its modules.

\begin{proposition} \label{prop:Aquo}
Let $M$ be an $A$-module. Then $M$ is a quotient of an $A$-module of the form $A \otimes V$ where $V$ is a polynomial representation of $\GL$. If $M$ is finitely generated, we can take $V$ to be finite length.
\end{proposition}

\begin{proof}
Since $M$ is an algebraic representation of $\Sp$, we can find a surjection $V \to M$ with $V$ a polynomial representation of $\GL$ (see \S \ref{ss:repSp}(\ref{sp:poly-quo})). We thus obtain a surjection of $A$-modules $A \otimes V \to M$. Now suppose $M$ is finitely generated, and let $W \subset M$ be a finite length $\Sp$-subrepresentation that generates it. Choose a surjection $V \to W$ with $V$ a finite length polynomial representation of $\GL$. Then the resulting map $A \otimes V \to M$ is surjective.
\end{proof}

The category $\Mod_A$ has no non-zero projective objects. In particular, the modules $A \otimes V$ appearing in the above proposition are not projective. (This is a consequence of the fact that $V$ is not a projective object of $\Rep(\Sp)$, hence any non-split surjection $W \to V$ cannot be lifted. Concretely, if $V=\bC$, then $\bigwedge^2 \bV \to \bC$ has no splitting.) However, these modules are $A$-flat, and so the proposition implies that $\Mod_A$ has enough flat objects. Thus there is no problem defining the functor $\Tor^A_{\bullet}(-, -)$ on $\Mod_A$.

\begin{proposition} \label{prop:forget-injective}
The forgetful functor $\Mod_A \to \Rep(\Sp)$ takes injective objects to injective objects.
\end{proposition}

\begin{proof}
It is right adjoint to the exact functor $\Rep(\Sp) \to \Mod_A$ given by $V \mapsto A \otimes V$.
\end{proof}

\begin{proposition} \label{prop:Anoeth}
$A$ is noetherian, that is, any submodule of a finitely generated module is finitely generated.
\end{proposition}

\begin{proof}
Let $\Phi \colon \Rep(\Sp) \to \Rep(\fS)$ be the restriction functor. Then $\Phi(A)$ is isomorphic to the algebra $R_2$ of \S \ref{ss:repSym}(\ref{sym:noeth}). If $M$ is a finitely generated $A$-module then $\Phi(M)$ is a finitely generated $R_2$-module: indeed, $M$ is a quotient of $A \otimes V$ for some finite length $\Sp$-representation $V$, and thus $\Phi(M)$ is a quotient of $R_2 \otimes \Phi(V)$, and $\Phi(V)$ is finite length as an $\fS$-representation by \S \ref{ss:repSym}(\ref{sym:sp-res}). Suppose now that $N_1 \subset N_2 \subset \cdots$ is an ascending chain of $A$-submodules of $M$. Then $\Phi(N_1) \subset \Phi(N_2) \subset \cdots$ is an ascending chain of $R_2$-submodules of $\Phi(M)$. Since $\Phi(M)$ is finitely generated over $R_2$ and $R_2$ is noetherian, the chain stabilizes. Thus the original chain stabilizes, as $\Phi$ does not affect the underlying vector space. This shows that $M$ is noetherian as an $A$-module.
\end{proof}

We say that an $A$-module $M$ is {\bf torsion} if every element is annihilated by a non-zero element of $A$. We now show that this notion is equivalent to two other reasonable notions of torsion.

\begin{proposition}
Let $M$ be a finitely generated $A$-module. Then the following conditions are equivalent:
\begin{enumerate}
\item $M$ is torsion in the above sense.
\item $M$ has finite length.
\item $M$ is annihilated by a power of $A_+$.
\end{enumerate}
\end{proposition}

\begin{proof}
It is clear that (b) and (c) are equivalent, and that both imply (a). Suppose now that $M$ satisfies (a). Let $x \in M$ be given and let $a \in A$ be a non-zero element such that $ax=0$. Let $V$ be the $\Sp$-subrepresentation of $M$ generated by $x$. Suppose that $E$ is an element of the Lie algebra $\fsp$. We then have $a(Ex)+(Ea)x=0$. Multiplying by $a$, we find that $a^2 (Ex)=0$. Continuing in this manner, we see that if $b$ is any element of $\cU(\fsp)$ then there is some $k$ such that $a^k (bx)=0$. It follows that every element of $V$ is annihilated by some power of $a$.

Let $n$ be such that $G_n$ fixes $a$, where $G_n$ is as in \S \ref{ss:repSp}(\ref{sp:sp}). (We note that any element of an algebraic representation of $\Sp$ is fixed by some $G_n$.) One easily sees that $V$ is finitely generated as a $G_n$-representation. Let $y_1, \ldots, y_r$ be generators, and let $k$ be such that $a^k y_i=0$ for all $1 \le i \le r$; this exists by the first paragraph. Since $G_n$ fixes $a$, it follows that $a^k y=0$ for any element $y \in V$. Since $V$ is $\Sp$-stable, it follows that $(ga^k) y=0$ for any $y \in V$ and any $g \in \Sp$. Thus if $I$ is the ideal of $A$ generated by the $\Sp$-orbit of $a^k$ then $IV=0$. Since each graded piece of $A$ is irreducible as an $\Sp$-representation (the symmetric power $\Sym^k \bV$ is equal to $\bS_{[k]} \bV$ since the invariants of $\bV^{\otimes k}$ are automatically traceless), the only $\Sp$-stable ideals of $A$ are powers of $A_+$ and the zero ideal. Since $I$ is not zero, we see that $A_+^n V=0$ for some $n$. Thus every element of $M$ is annihilated by a power of $A_+$. Since $M$ is finitely generated, it follows that $M$ is annihilated by a power of $A_+$. Thus (c) holds.
\end{proof}

\begin{proposition} \label{prop:Aspec}
Let $M$ be a finitely generated $A$-module. Then $\ell(M)$ is finite and $\rR^i \Gamma_n(M)=0$ for all $n \ge \ell(M)$ and $i>0$.
\end{proposition}

\begin{proof}
Write $M$ as a quotient of $A \otimes V$ for some finite length $\Sp$-representation $V$. Then $\ell(M) \le \ell(A \otimes V)$ by the definition and $\ell(A \otimes V) \le \ell(A) + \ell(V)$ by \S \ref{ss:repSp}(\ref{sp:len}). Since $\ell(A)=1$ (since $\Sym^k \bV = \bS_{[k]} \bV$) and $\ell(V)$ is finite, as $V$ has finite length, it follows that $\ell(M)$ is finite. The vanishing statement now follows from \S \ref{ss:repSp}(\ref{sp:sp}).
\end{proof}

\begin{corollary} \label{cor:sp-surj}
Let $M \to N$ be a surjection of finitely generated $A$-modules. Then $\Gamma_n(M) \to \Gamma_n(N)$ is surjective for $n \gg 0$.
\end{corollary}

\begin{proof}
Let $K$ be the kernel of $M \to N$, which is finitely generated, and simply take $n \ge \ell(K)$; since $\rR^1 \Gamma_n(K)=0$, the result follows.
\end{proof}

\section{Representations of $H$} \label{s:repH}

Let $\xi \colon \bV \to \bC$ be the linear form defined by $\xi(e_i)=\xi(f_i)=1$ for all $i$. Let $H$ be the subgroup of $\Sp$ that stabilizes $\xi$. We say that a representation of $H$ is {\bf algebraic} if it occurs as a subquotient of a direct sum of tensor powers of $\bV$. We let $\Rep(H)$ denote the category of algebraic representations of $H$. In this section, we determine the structure of this category.

\subsection{Weyl's construction}

Let $\bW = \ker \xi$.  Given a positive integer $n$, we let $V_n$ denote the subspace spanned by $e_1,f_1,\dots,e_n,f_n$ and let $W_n$ denote the kernel of $\xi$ restricted to $V_n$. Similarly, $H_n \subset \Sp(V_n)$ is the stabilizer of $\xi$. The symplectic form $\omega \colon \lw^2\bV \to \bC$ restricts to an alternating form $\omega \colon \lw^2 \bW \to \bC$. Working with finitely many variables, the radical of $\omega$ on $W_n$ is precisely the span of $\sum_{i=1}^n (e_i-f_i)$, and we denote it by $W_n^\perp$.

\begin{proposition} \label{prop:kernel}
Let $d$ be a positive integer. Every nonzero $H_n$-submodule of $W_n^{\otimes d}$ has a nonzero intersection with the kernel of $\pi_n \colon W_n^{\otimes d} \to (W_n/W_n^\perp)^{\otimes d}$. 
\end{proposition}

\begin{proof}
  If $n=1$, this is clear, so assume $n \ge 2$.  Define a new symplectic basis of $V_n$ by
  \[
    v_1 = \frac{1}{2n} \sum_{i=1}^n (e_i+f_i), \qquad w_1 = \sum_{i=1}^n (f_i-e_i), \qquad v_i = e_1 - e_i, \qquad w_i = f_1 - f_i \qquad (i \ge 2).
  \]
  Then $\xi(v_1)=1$ while $\xi$ is 0 on all other basis vectors, so $W_n$ is the span of $v_2,\dots,v_n,w_1,\dots,w_n$. We will write all elements of $W_n^{\otimes d}$ in terms of the basis given by tensor products of the symplectic basis we just specified. Then $\ker \pi_n$ consists of those vectors spanned by tensors which have at least one instance of $w_1$.

  Let $U$ be a nonzero $H_n$-submodule of $W_n^{\otimes d}$, pick a nonzero vector $u \in U$, and expand it in the basis mentioned above. If $u \notin \ker \pi_n$, then there is a basis element which has a nonzero coefficient for $u$ such that either it
  \begin{enumerate}
  \item contains $v_i$ as a tensor factor for some $i$ and does not contain $w_1$ as a tensor factor, or 
  \item contains $w_i$ as a tensor factor for some $i \ge 2$ and does not contain $w_1$ as a tensor factor.
  \end{enumerate}
  In case (a), define $g \in H_n$ by $g(v_1)=v_1 + w_i$, $g(v_i) = w_1+ v_i$, and $g$ fixes all other basis vectors. Then $g(u) - u \in U \cap \ker \pi_n$, so it suffices to prove that $g(u)\ne u$. The $g$ we defined is the upper-triangular unipotent element in a copy of $\GL_2$ acting on the span of $w_1$ and $v_i$, and so being fixed by $g$ is the same as being a sum of highest weight vectors for this group. However, since we have a basis vector in $u$ that contains $v_i$ but not $w_1$, it is a sum of weight vectors where each weight is of the form $\begin{pmatrix} t_1 & 0 \\ 0 & t_2 \end{pmatrix} \mapsto t_2^r$ with $r > 0$, which is not a dominant weight. Hence we conclude that $g(u)-u \ne 0$.

  Case (b) is similar, we instead define $g \in H_n$ by $g(v_1) = v_1 - v_i$, $g(w_i) = w_1+w_i$ and $g$ fixes all other basis vectors.
\end{proof}

For each $1\le i<j \le d$, we have a contraction map $\bW^{\otimes d} \to \bW^{\otimes (d-2)}$ which applies $\omega$ to the $i$th and $j$th tensor factors and we let $\bW^{[d]}$ denote the kernel over all choices of $i<j$. For each partition $\lambda$ of $d$, choose an embedding of the Schur functor $\bS_\lambda(\bW) \subset \bW^{\otimes d}$ and set $\bS_{[\lambda]} \bW = \bS_\lambda(\bW) \cap W^{[d]}$. For each $n$, we can also define $\bS_{[\lambda]} W_n$, which is nonzero for $n \gg 0$.

\begin{proposition} \label{prop:weyl-irred}
$\bS_{[\lambda]}\bW$ is an irreducible $H$-module.
\end{proposition}

\begin{proof}
  Let $U \subset \bS_{[\lambda]} \bW$ be a nonzero $H$-submodule. Let $U_n = U \cap \bS_{[\lambda]} W_n$. We have a quotient map $\pi_n \colon \bS_{[\lambda]} W_n \to \bS_{[\lambda]} (W_n/W_n^\perp)$. By Proposition~\ref{prop:kernel}, $U_n \cap \ker \pi_n \ne 0$ for $n \gg 0$. Pick a nonzero vector $u$ in the intersection. Then $u$ also belongs to $U_{n+1}$. Using the notation from the proof of Proposition~\ref{prop:kernel}, let $v^{(n)}_i, w^{(n)}_i$ be the basis defined for $V_n$. Then $w^{(n)}_1 = (w^{(n+1)}_1 - v^{(n+1)}_{n+1} + w^{(n+1)}_{n+1})/2$, so in particular, $u \notin \ker \pi_{n+1}$. To see this, we first embed $\bS_{[\lambda]} W_{n+1} \subset W_{n+1}^{\otimes d}$. Since $u \in \ker \pi_n$, it means that when written in the basis for $V_n$, all basis vectors with nonzero coefficient have $w_1^{(n)}$. When we expand in the basis for $V_{n+1}$, the sum of basis vectors which have $w^{(n+1)}_{n+1}$ but not $w^{(1)}_{n+1}$ or $v^{(n+1)}_{n+1}$ is the result of replacing $w^{(1)}_n$ by $w^{(n+1)}_{n+1}/2$ in $u$, which is nonzero.

  Next, $\bS_{[\lambda]}(W_{n+1}/W_{n+1}^\perp)$ is an irreducible representation of $\Sp(W_{n+1}/W_{n+1}^\perp)$, and hence is irreducible for $H_{n+1}$, so that  $\pi_{n+1}(U_{n+1}) = \bS_{[\lambda]}(W_{n+1}/W_{n+1}^\perp)$. We claim that this implies that $U_{n+1} = \bS_{[\lambda]} W_{n+1}$.

  Pick a vector $x \in W_{n+2}$ so that
  \[
    v_2^{(n+1)}, w_2^{(n+1)}, \dots, v_{n+1}^{(n+1)}, w_{n+1}^{(n+1)}, w_1^{(n+1)}, x
  \]
  is a symplectic basis for the space $W'$ that it spans. This choice of basis determines a Borel subgroup $B' \subset \Sp(W')$. Taking just the first $2n$ vectors gives a basis for $W_{n+1}/W_{n+1}^\perp$, and hence determines a Borel subgroup $B \subset \Sp(W_{n+1}/W_{n+1}^\perp)$. Choose a $\Sp(W_{n+1}/W_{n+1}^\perp)$-equivariant splitting $\psi$ of $\pi_{n+1}$. Let $\alpha$ be a highest weight vector in $\bS_{[\lambda]}(W_{n+1}/W_{n+1}^\perp)$ with respect to $B$. Then $\psi(\alpha)$ is a highest weight vector in $\bS_{[\lambda]} W'$ (this follows from the fact that it is a weight vector of weight $\lambda$; this weight space is 1-dimensional in $\bS_{[\lambda]} W'$), so by irreducibility, we see that $U \cap \bS_{[\lambda]} W' = \bS_{[\lambda]} W'$, and hence the claim is proven. Since this is true for all $n \gg 0$, we conclude that $U = \bS_{[\lambda]} \bW$.
\end{proof}

\begin{proposition} \label{prop:tensor-simples}
  The irreducible constituents of $\bW^{\otimes d}$ are $\bS_{[\lambda]} \bW$ where $|\lambda|\le d$ and $d-|\lambda|$ is even.
\end{proposition}

\begin{proof}
The simple constituents of $\bW^{[d]}$ are $\bS_{[\lambda]} \bW$ where $|\lambda|=d$. Next, each contraction map $\bW^{\otimes d} \to \bW^{\otimes d-2}$ is surjective, so the rest follows by induction. 
\end{proof}

\subsection{Diagrammatic description}

Let $\cC$ be the following $\bC$-linear category. The objects are finite sets. The space $\Hom_{\cC}(S, T)$ is spanned by pairs $(f, \Gamma)$ where $f \colon S \to T$ is an injection and $\Gamma$ is a partial directed matching on $T \setminus f(S)$. If $\Gamma$ is obtained from $\Gamma'$ by flipping the orientation of a single edge then $(f,\Gamma')$ is identified with $-(f,\Gamma)$ in $\Hom_{\cC}(S,T)$, and these relations generate all relations. We let $\Mod_{\cC}$ denote the category of $\cC$-modules, i.e., the category of $\bC$-linear functors $\cC \to \Vec$, and write $\Mod_{\cC}^{\lf}$ for the full subcategory spanned by locally finite length objects.

We define a $\bC$-linear functor $\cK \colon \cC^\op \to \Rep(H)$ by $\cK(S) = \bV^{\otimes S}$. Given a morphism $(f,\Gamma) \colon S \to T$, we define $\bV^{\otimes T} \to \bV^{\otimes S}$ as follows. The injection $f$ identifies $S$ with a subset of $T$ and we map the corresponding tensor factors indexed by elements in $T$ to those they correspond to in $S$. If two elements $x,y \in T$ are connected by an edge of the partial matching with orientation $x \to y$, we apply $\omega$ to those two factors with the vector in position $x$ placed in the first argument. For all other factors, we apply $\xi$.

For the definition of the tensor product $\otimes^\cC$, see \cite[2.1.9]{infrank}.

\begin{theorem} \label{thm:Hrep}
The functor $\Mod_{\cC}^{\lf} \to \Rep(H)$ given by $M \mapsto \cK \otimes^{\cC} M$ is an equivalence of categories.
\end{theorem}

\begin{proof}
Since the functor in question is cocontinuous, and each category is locally noetherian and artinian, it suffices to check that it induces an equivalence on the categories of finite length objects. For this, we apply \cite[Theorem 2.1.11]{infrank} and its corollary. Criterion (a) is Proposition~\ref{prop:weyl-irred}.

Now we verify criterion (b). Proposition~\ref{prop:tensor-simples} shows that $\bS_{[\lambda]} \bW$ has no nonzero maps to $\bW^{\otimes d}$ if $d<|\lambda|$. So we have to show that the same is true if $d>|\lambda|$. Consider the span $\bW'$ of $\{e_1-e_2,e_1-e_3,\dots,f_1-f_2,f_1-f_3,\dots\}$ and let $T \subset \Sp(\bW')$ be a maximal torus with respect to this basis. This consists of maps $e_1-e_i \mapsto \alpha_i (e_1-e_i)$ and $f_1-f_i \mapsto \alpha_i^{-1} (f_1-f_i)$. We define the magnitude of a weight $(\alpha_1,\dots) \mapsto \prod_i \alpha_i^{n_i}$ to be $\sum_i |n_i|$. The proof of Proposition~\ref{prop:weyl-irred} shows that every submodule of $\bW^{\otimes d}$ has a weight vector whose weight has magnitude $d$, which proves what we want.
\end{proof}

\begin{corollary}
The injective envelope of $\bS_{[\lambda]}\bW$ is $\bV_{\lambda}$. The $\bV_{\lambda}$ account for all the indecomposable injectives of $\Rep(H)$.
\end{corollary}

\begin{proof}
  Let $B\fS_n$ be the category with one object $\ast$ with $\Hom(\ast,\ast)=\fS_n$, the symmetric group on $n$ letters. Let $i \colon B\fS_n \to \cC^\op$ be the functor taking $\ast$ to the set $[n]=\{1,\ldots,n\}$, and acting in the obvious manner on morphisms. The pullback functor $i^* \colon \Mod_{\cC^\op} \to \Rep(\fS_n)$ has a left adjoint $i_\#$, called the left Kan extension. Let $\cM_{\lambda}$ be the Specht module for $\fS_n$ associated to the partition $\lambda$. Since $\cM_{\lambda}$ is an projective object of $\Rep(\fS_n)$, it follows that $i_\#(\cM_{\lambda})$ is a projective object of $\Mod_{\cC^\op}$. One easily sees that it is indecomposable, and that every indecomposable projective object of $\Mod_{\cC^\op}^{\lf}$ has this form, for a unique $\lambda$.

  By \cite[2.1.10]{infrank}, we have a contravariant equivalence $\Phi \colon \Mod_{\cC^\op}^\lf \to \Rep(H)$ given by $M \mapsto \hom_{\cC^\op}(M,\cK)$, so that $\Phi(i_\#(\cM_\lambda))$ is the injective envelope of $\bS_{[\lambda]} \bW$. We have
\begin{displaymath}
 \Hom_{\cC^\op}(i_\#(\cM_{\lambda}), \cK)
= \Hom_{\fS_n}(\cM_{\lambda}, \cK([n]))
= \Hom_{\fS_n}(\cM_\lambda, \bV^{\otimes n})
= \bV_{\lambda},
\end{displaymath}
which proves the corollary.
\end{proof}

\begin{corollary}
Every finitely generated object of $\Rep(H)$ has finite injective dimension.
\end{corollary}

\begin{proof}
This is clearly true in $\Mod_{\cC}$: every indecomposable injective object $\bV_\lambda$ has finite length and is supported in degrees $\le |\lambda|$.
\end{proof}

\subsection{Description via tca's}

Let $B$ be the tca $\Sym(\bV \oplus \lw^2{\bV})$. The following is an analog of \S \ref{ss:repSp}(\ref{sp:tca}).

\begin{theorem} \label{thm:Htca}
We have a natural equivalence of categories $\Phi \colon \Rep(H) \cong \Mod_B^{\lf}$. This equivalence satisfies $\Phi(\bV_{\lambda})=I_B(\bV_{\lambda})$ and $\Phi(\bS_{[\lambda]}\bW)=\bV_{\lambda}$ (with $B_+$ acting by~$0$).
\end{theorem}

\begin{proof}
The construction $\Phi$, and the fact that it is an equivalence, is exactly analogous to \cite[Theorem~4.3.1]{infrank}. It is immediate from the construction that $\Phi$ takes the simple $\bS_{[\lambda]} \bW$ to the simple $\bV_{\lambda}$. Since $\bV_{\lambda}$ is the injective envelope of $\bS_{[\lambda]}$ in $\Rep(H)$ and $I_B(\bV_{\lambda})$ is the injective envelope of $\bV_{\lambda}$ in $\Mod_B$, it follows that $\Phi(\bV_{\lambda}) = I_B(\bV_{\lambda})$.
\end{proof}

\begin{corollary}
We have
\begin{displaymath}
\dim \Ext^i_H(\bS_{[\lambda]}\bW, \bS_{[\mu]}\bW)
= \dim \Hom_{\GL}(\lw^i(\bV \oplus \lw^2{\bV}) \otimes \bV_{\lambda}, \bV_{\mu}).
\end{displaymath}
\end{corollary}

\begin{proof}
By Theorem~\ref{thm:Htca}, we have
\begin{displaymath}
\Ext^i_H(\bS_{[\lambda]}\bW, \bS_{[\mu]}\bW)
= \Ext^i_B(\bV_{\lambda}, \bV_{\mu}).
\end{displaymath}
This can be computed using the Koszul resolution, which yields the stated result.
\end{proof}

\subsection{An equality of dimensions}

We will require the following result in \S \ref{s:gen}.

\begin{proposition} \label{prop:dimcount}
For partitions $\lambda$ and $\mu$, we have
\begin{displaymath}
\dim \Hom_{\Sp}(\bV_{\lambda}, \bV_{\mu}(\bV) \otimes \Sym(\bV)) =
\dim \Hom_H(\bV_{\lambda}, \bV_{\mu})
\end{displaymath}
\end{proposition}

\begin{proof}
We have seen that $\Rep(H)$ is equivalent to $\Mod_B^{\lf}$ where $B$ is the tca $\Sym(\bV \oplus \lw^2{\bV})$, and the injective $\bV_{\lambda}$ in $\Rep(H)$ corresponds to the injective $I_B(\bV_{\lambda})$ in $\Mod_B$. Write $B=B_1 \otimes B_2$, where $B_1=\Sym(\bV)$ and $B_2=\Sym(\lw^2{\bV})$. Then by \S \ref{ss:repSp}(\ref{sp:tca}), $\Rep_{\Sp}$ is equivalent to $\Mod_{B_2}^{\lf}$, and the injective $\bV_{\lambda}$ in $\Rep(\Sp)$ corresponds to the injective $I_{B_2}(\bV_{\lambda})$ in $\Mod_{B_2}$. We therefore have
\begin{align*}
\dim \Hom_{\Sp}(\bV_{\lambda}, \bV_{\mu} \otimes \Sym(\bV))
&= \dim \Hom_{B_2}(I_{B_2}(\bV_{\lambda}), I_{B_2}(\bV_{\mu} \otimes \Sym(\bV))) \\
&= \dim \Hom_{B_2}(P_{B_2}(\bV_{\lambda}), P_{B_2}(\bV_{\mu} \otimes \Sym(\bV))) \\
&= \dim \Hom_{B_2}(P_{B_2}(\bV_{\lambda}), P_B(\bV_{\mu})) \\
&= \dim \Hom_B(P_B(\bV_{\lambda}), P_B(\bV_{\mu})) \\
&= \dim \Hom_B(I_B(\bV_{\lambda}), I_V(\bV_{\mu})) \\
&= \dim \Hom_H(\bV_{\lambda}, \bV_{\mu}).
\end{align*}
In the first step, we used the equivalence $\Rep(\Sp)=\Mod_{B_2}^{\lf}$; in the second Proposition~\ref{prop:inj-proj}; in the third, the identification
\begin{displaymath}
P_{B_2}(\bV_{\mu} \otimes \Sym(\bV))=B_2 \otimes \bV_{\mu} \otimes \Sym(\bV) = B \otimes \bV_{\mu} = P_B(\bV_{\mu});
\end{displaymath}
in the fourth, that extension of scalars is left adjoint to restriction of scalars; in the fifth, Proposition~\ref{prop:inj-proj}; and in the final step, the equivalence $\Rep(H) \cong \Mod_B^{\lf}$.
\end{proof}

\subsection{The dependence on $\xi$} \label{ss:xi}

Suppose that $\xi' \colon \bV \to \bC$ is another non-zero linear form, and let $H' \subset \Sp$ be its stabilizer. It is natural to ask if $\Rep(H')$ is equivalent to $\Rep(H)$.

In fact, this is not true in general. Indeed, take $\xi'$ to be the linear functional given by $\xi'(v)=\langle e_1, v \rangle$.  Then $H'$ is just the stabilizer of $e_1$. It follows that the map $\bC \to \bV$ taking~1 to $e_1$ is a map of $H'$ representations; moreover, it lands in the kernel of $\xi'$. One easily sees that $\Hom_{H'}(\bV, \bC)$ is one-dimensional and spanned by $\xi'$. Thus the map $\bC \to \bV$ has no section, and so $\bC$ is not injective in $\Rep(H')$. Since $\bC$ is distinguished as the unit object for the tensor structure, it follows that there is no equivalence of tensor categories $\Rep(H) \cong \Rep(H')$. With more effort, one can show there is no equivalence at all.

Let $\iota \colon \bV \to \bV^*$ be the map $\iota(v)=\langle v, - \rangle$. The reasoning of the previous paragraph applies whenever $\xi'$ belongs to $\im(\iota)$: for such functionals, $\Rep(H')$ is not equivalent to $\Rep(H)$. We believe that if $\xi' \not\in \im(\iota)$ then $\Rep(H')$ is equivalent to $\Rep(H)$, but we have not investigated this carefully.

\section{The local structure of $A$-modules} \label{s:loc}

\subsection{Statement of results}

Recall that $A$ is the polynomial ring in variables $\{x_i,y_i\}_{i \ge 1}$. Let $\fm$ be the ideal of $A$ generated by $x_i-1$ and $y_i-1$ for $i \ge 1$. Note that $\fm$ is the kernel of the ring homomorphism $A \to \bC$ induced by the linear functional $\xi$. The following is the main result of this section:

\begin{theorem} \label{thm:loc}
Let $M$ be an $A$-module. Then there is a canonical and functorial isomorphism $M_{\fm} \to M/\fm M \otimes_{\bC} A_{\fm}$ of $\vert A_{\fm} \vert$-modules. In particular, $M_{\fm}$ is a free $\vert A_{\fm} \vert$-module.
\end{theorem}

The proof of Theorem~\ref{thm:loc} will take the remainder of this section. The arguments are similar to (but somewhat easier than) those used in \cite[\S 3.2]{sym2noeth} and \cite[\S 5]{periplectic} to analyze analogous generic categories.

\subsection{The group $K$}

Let $K_i$ be the upper triangular Borel subgroup of $\Sp_2$, regarded as a subgroup of $\Sp$ by acting on $e_i$ and $f_i$. We represent elements of $K_i$ as matrices
\begin{displaymath}
\begin{pmatrix} t_i & u_i \\ 0 & t_i^{-1} \end{pmatrix},
\end{displaymath}
and identify the coordinate ring $\bC[K_i]$ of $K_i$ with $\bC[t_i^{\pm 1}, u_i]$. We let $K$ be the product of the $K_i$'s. This is the affine group scheme with coordinate ring $\bC[K]=\bC[t_i^{\pm 1}, u_i]_{i \ge 1}$.

Suppose that $V$ is an algebraic representation of $\Sp$ and $x \in V$. Then for $i \gg 0$ the subgroup $K_i$ of $\Sp$ fixes $x$. It follows that the group $K$ naturally acts on $V$; in other words, we can restrict algebraic representations of $\Sp$ to $K$, even though $K$ is not quite a subgroup of $\Sp$. The $K$-subrepresentation of $V$ generated by $x$ is easily seen to be finite dimensional. It follows that $V$ is naturally a comodule over $\bC[K]$, that is, we have a natural comultiplication map
\begin{displaymath}
V \to V \otimes \bC[K].
\end{displaymath}
Explicitly, this map takes $v \in V$ to the function $K \to V$ given by $k \mapsto kv$.

The group $K_i$ is contained in $\Sp$, and so its Lie algebra $\fk_i$ is contained in $\fsp$. We let $\fk = \sum_{i \ge 1} \fk_i$, which is a Lie subalgebra of $\fsp$. This is not quite the Lie algebra of $K$---$\operatorname{Lie}(K)$ is the product of the $\fk_i$'s---but the difference is negligible for our purposes.

We let $\fh_n$ be the Lie algebra of $H_n$ which we think of as a subalgebra of $\fsp$, and set $\fh = \sum_n \fh_n$.

\begin{proposition} \label{prop:sp=k+h}
We have $\fsp=\fk \oplus \fh$.
\end{proposition}

\begin{proof}
We regard elements of $\fsp$ as $\infty \times \infty$ matrices by using the ordered basis $e_1,f_1,e_2,f_2,\ldots$ of $\bV$. For an $\infty \times \infty$ matrix $m$, we let $m_i$ be the $2 \times \infty$ matrix whose rows are given by the $2i-1$ and $2i$ rows of $m$. Let $X \in \fsp$ be given. Note that $X_i$ has finitely many non-zero entries and $X_i=0$ for $i \gg 0$. There is a unique matrix $Y(i)$ of the form
\begin{displaymath}
\begin{pmatrix}
\cdots & 0 & t & u & 0 & \cdots \\
\cdots & 0 & 0 & -t & 0 & \cdots
\end{pmatrix}
\end{displaymath}
where the $t$ is in column $2i-1$ such that each row of $X_i+Y(i)$ sums to zero. For $i \gg 0$ we have $X_i=0$ and so $Y(i)=0$. Let $Y \in \fk$ be the matrix with $Y_i=Y(i)$ for all $i$. Then each row of $X+Y$ sums to~0, and so $X+Y \in \fh$. Clearly, $Y$ is the unique element of $\fb$ with this property: indeed, if $Y'$ were a second such element then for all $i$ we would have $X_i+Y'_i=0$, and thus $Y'_i=Y(i)=Y_i$, and thus $Y'=Y$. The result thus follows.
\end{proof}

\subsection{The map $\phi$}

Let $M$ be an $A$-module. We have the comultiplication map $M \to M \otimes \bC[K]$ discussed above. Composing this with the quotient map $M \to M/\fm M$, we obtain a linear map
\begin{displaymath}
\phi_M \colon M \to M/\fm M \otimes \bC[K].
\end{displaymath}
In fact, we can define $\phi_M$ for any $K$-equivariant $\vert A \vert$-module (where ``$K$-equivariant'' means the action of $K$ comes from a comodule structure). We now study this map. We begin with the case $M=A$:

\begin{proposition} \label{prop:phiA}
We have the following:
\begin{enumerate}
\item The map $\phi_A \colon A \to \bC[K]$ is the $\bC$-algebra homomorphism given by
\begin{displaymath}
\phi_A(x_i)=t_i, \qquad \phi_A(y_i)=t_i^{-1}+u_i.
\end{displaymath}
\item The extension $\fn$ of $\fm$ along $\phi_A$ is the maximal ideal of $\bC[K]$ generated by $t_i-1$ and $u_i$ for $i \ge 1$.
\item The map $\phi_A$ induces an isomorphism of localizations $A_{\fm} \to \bC[K]_{\fn}$.
\end{enumerate}
\end{proposition}

\begin{proof}
Since $K$ acts on $A$ by $\bC$-algebra homomorphisms, the map $\phi_A$ is an algebra homomorphism. The group $K_j$ fixes $x_i$ and $y_i$ for $i \ne j$. The action of $K_i$ is given by
\begin{displaymath}
\begin{pmatrix} t_i & u_i \\ 0 & t_i^{-1} \end{pmatrix} x_i = t_i x_i, \qquad
\begin{pmatrix} t_i & u_i \\ 0 & t_i^{-1} \end{pmatrix} y_i = t_i^{-1} y_i+u_i x_i.
\end{displaymath}
We thus see that the comultiplication map $A \to A \otimes \bC[K]$ takes $x_i$ to $x_i \otimes t_i$ and $y_i$ to $y_i \otimes t_i^{-1} + x_i \otimes u_i$. Since the map $A \to A/\fm=\bC$ takes $x_i$ and $y_i$ to~1, we find $\phi_A(x_i)=t_i$ and $\phi_A(y_i)=t_i^{-1}+u_i$, which proves~(a). Statement~(b) now follows. We now prove (c). Define $\psi \colon \bC[K] \to A_{\fm}$ to be the algebra homomorphism given by $\psi(t_i)=x_i$ and $\psi(u_i)=y_i-x_i^{-1}$. The kernel of the composition $\bC[K] \to A_{\fm} \to A_{\fm}/\fm A_{\fm}$ is the ideal $\fn$, from which it follows that $\psi^{-1}(\fm A_{\fm})=\fn$. Thus $\psi$ extends to a ring homomorphism $\bC[K]_{\fn} \to A_{\fm}$, which is clearly the inverse to the localization of $\phi_A$.
\end{proof}

In what follows, we regard $\bC[K]$ as an $A$-module via $\phi_A$. We now study the map $\phi_M$ for an arbitrary $A$-module $M$. Since $K$ acts on $M$ by $A$-semilinear automorphisms, it follows that $\phi_M$ is a morphism of $A$-modules. Our goal is to prove that $\phi_M$ induces an isomorphism after localizing at $\fm$, which will establish the theorem. We require some lemmas first.

\begin{lemma} \label{lem:phi-modm}
Let $M$ be a $K$-equivariant $\vert A \vert$-module. Then $\phi_M$ induces an isomorphism modulo $\fm$.
\end{lemma}

\begin{proof}
Consider the following diagram:
\begin{displaymath}
\xymatrix@C=2cm{
M \ar[r]^-{\Delta} \ar[rd]_-{\phi_M} & M \otimes \bC[K] \ar[r]^{1 \otimes \eta} \ar[d] & M \ar[d] \\
& M/\fm M \otimes \bC[K] \ar[r]^{1 \otimes \eta} & M/\fm M }
\end{displaymath}
Here $\Delta$ is the comultiplication map and $\eta$ is the natural map $\bC[K] \to \bC[K]/\fn \bC[K]=\bC$. We note that $\fn \subset \bC[K]$ corresponds to the identity element of $K$, and thus $\eta$ is the counit of the Hopf algebra structure. Thus the first line above composes to $\id_M$. It follows that the composition of $\phi_M$ with the morphism in the bottom row is the natural map $M \to M/\fm M$, which proves the result.
\end{proof}

\begin{lemma} \label{lem:phi-locm}
Let $M$ be an $A$-module. Then the localization of $\phi_M$ at $\fm$ is surjective.
\end{lemma}

\begin{proof}
Let $V$ be a finite dimensional $B$-subrepresentation of $M$. Let $N(V)$ be the $\vert A \vert$-submodule of $M$ generated by $V$. Since $V$ is finite dimensional, $N(V)$ is finitely generated as a $\vert A \vert$-module; in particular, $N(V)/\fm N(V)$ is finite dimensional. It follows by Nakayama's lemma that $(\phi_{N(V)})_{\fm}$ is surjective, since it is surjective modulo $\fm$ by Lemma~\ref{lem:phi-modm}.

Now, the following diagram commutes:
\begin{displaymath}
\xymatrix@C=2cm{
N(V) \ar[r]^-{\phi_{N(V)}} \ar[d] & N(V)/\fm N(V) \otimes \bC[K] \ar[d] \\
M \ar[r]^-{\phi_M} & M/\fm M \otimes \bC[K] }
\end{displaymath}
The vertical maps here are induced by the inclusion $N(V) \subset M$. Let $x \in M$ and let $y \in \bC[K]_{\fm}$. Let $V$ be the $K$-subrepresentation of $M$ generated by $x$, which is finite dimensional. Then $x \otimes y \in N(V)/\fm N(V) \otimes \bC[B]_{\fm}$ maps to $x \otimes y \in M/\fm M \otimes \bC[K]_{\fm}$ under the above map. Since the former element belongs to the image of $(\phi_{N(V)})_{\fm}$, the latter element belongs to the image of $(\phi_M)_{\fm}$. The result follows.
\end{proof}

\begin{remark}
In the proof of \cite[Proposition~5.9]{periplectic}, we said that $(\phi_M)_{\fm}$ was surjective simply because it was surjective modulo $\fm$ and its target is locally free at $\fm$. This seems inadequate. The reasoning in the above lemma applies to the situation in loc.\ cit., and thus fixes this gap.
\end{remark}

\begin{lemma}
Let $\fg$ be a Lie algebra and let $\fk$ and $\fh$ be subalgebras such that $\fg=\fk+\fh$. Suppose that $R$ is a commutative ring on which $\fg$ acts, and $\fa$ is an ideal of $R$ that is stable by $\fh$. Let $M$ be a $\fg$-equivariant $R$-module and let $N$ be the maximal $\fk$-submodule of $\fa M$. Then $N$ is $\fg$-stable.
\end{lemma}

\begin{proof}
It is clear that $N$ is $\fk$-stable, so we must show it is $\fh$-stable. Thus let $X \in \fh$ be given; we show that $XN \subset N$. Note that $N$ consists of all $n \in M$ such that $Y_1 \cdots Y_r n \in \fa M$ for all $Y_1, \ldots, Y_r \in \fk$. We must therefore show that $Y_1 \cdots Y_r Xm \in \fa M$ for all $Y_1, \ldots, Y_r \in \fk$ and $m \in N$. Write $[Y_r, X]=X'+Y'$ with $X' \in \fh$ and $Y' \in \fk$. Then
\begin{displaymath}
Y_1 \cdots Y_r X m = Y_1 \cdots Y_{r-1} (X Y_r m + X' m + Y' m).
\end{displaymath}
Since $N$ is $\fk$-stable, $Y_r m$ belongs to $N$, and so $Y_1 \cdots Y_{r-1} X(Y_r m)$ belongs to $\fa M$ by induction on $r$. Similarly, $Y_1 \cdots Y_{r-1} X' m$ belongs to $\fa M$ by induction on $r$. The term $Y_1 \cdots Y_{r-1} Y' m$ belongs to $N$ since $N$ is $\fk$-stable. Thus the result follows.
\end{proof}

\begin{lemma} \label{lem:ker}
Let $M$ be an $A$-module. Then the kernel of $\phi_M$ is $\Sp$-stable, and thus is an $A$-submodule of $M$.
\end{lemma}

\begin{proof}
  The kernel of $\phi_M$ consists of those elements $m \in M$ such that $km \in \fm M$ for all $k \in K$; thus $\ker(\phi_M)$ is the maximal $\fk$-submodule of $\fm M$. By Proposition~\ref{prop:sp=k+h}, we have $\fsp = \fk \oplus \fh$, and $\fm$ is stable under $\fh$. Hence the result follows from the previous lemma.
\end{proof}

By an ``algebraically $H$-equivariant $A_{\fm}$-module,'' we mean an $A_{\fm}$-module $N$ equipped with a compatible action of $H$ such that for every $x \in N$ there is a unit $u \in A_{\fm}$ such that $ux$ generates an algebraic $H$-representation. The localization of any $A$-module at $\fm$ is an algebraically $H$-equivariant $A_{\fm}$-module. Any submodule or quotient module of an algebraically $H$-equivariant $A_{\fm}$-module (in the category of equivariant modules) is again algebraically equivariant.

\begin{lemma} \label{lem:eqfg}
Suppose that
\begin{displaymath}
0 \to R \to M \to N \to 0
\end{displaymath}
is an exact sequence of algebraically $H$-equivariant $A_{\fm}$-modules such that $M$ is equivariantly finitely generated and $N$ is free as an $\vert A_{\fm} \vert$-module. Then $R$ is also equivariantly finitely generated.
\end{lemma}

\begin{proof}
The argument in \cite[Lemma~5.8]{periplectic} applies here.
\end{proof}

\begin{lemma} \label{lem:nak}
Let $M$ be a finitely generated $A$-module such that $M=\fm M$. Then $M_{\fm}=0$.
\end{lemma}

\begin{proof}
Let $G_n \subset \Sp$ be as in \S \ref{ss:repSp}(\ref{sp:sp}), let $\fS \subset \Sp$ be the embedding of the infinite symmetric group into $\Sp$ as in \S \ref{ss:repSym}(\ref{sym:sp-res}), and let $\fS_{>n}$ be as in \S \ref{ss:repSym}(\ref{sym:large-n}). Note that $\fS_{>n}=\fS \cap G_n$ and $\fS_n=\fS \cap \Sp_{2n}$.

Let $V$ be a finite length $\Sp$-subrepresentation of $M$ that generates $M$ as an $A$-module. Pick $m_1, \ldots, m_r \in V$ such that the $m_i$ generate $V^{\fS_{>n}}$ as an $\fS_n$-representation for all $n \gg 0$; this is possible by \S \ref{ss:repSym}(\ref{sym:large-n}). Note that the $m$'s then generate $V$ as an $\fS$-representation. Also note that, for $n$ large, the elements $m_1, \ldots, m_r$ belong to $V^{G_n} \subset V^{\fS_{>n}}$, which is an $\fS_n$-subrepresentation, and so $V^{G_n}=V^{\fS_{>n}}$.

For $1 \le i \le r$, write $m_i=\sum_j a_{i,j} n_{i,j}$ with $a_{i,j} \in \fm$ and $n_{i,j} \in M$. Let $n \gg 0$ be sufficiently large so that the $a_{i,j}$ belong to $A'=A^{G_n}$ and the $n_{i,j}$ belong to $M'=M^{G_n}$. Let $V'=V^{G_n}$ and let $\fm'=\fm \cap A'$. By Corollary~\ref{cor:sp-surj}, the map $A' \otimes V' \to M'$ is surjective (after possibly enlarging $n$), and so $V'$ generates $M'$ as an $A'$-module. We have $m_i \in \fm' M'$ for all $i$, and so $g m \in \fm' M'$ for all $g \in \fS_n$ since $\fm'$ is $\fS_n$-stable. Thus $V' \subset \fm' M'$, and so $M'=\fm' M'$. Thus, by the standard version of Nakayama's lemma, we have $M'$ localizes to~0 at $\fm'$. Therefore, for each $1 \le i \le k$, there is some $s_i \in A' \setminus \fm'$ such that $s_i m_i=0$. For any $g \in \fS$, we have $gs_i \in A \setminus \fm$, and so $(gs_i)(gm_i)=0$. It follows that $gm_i$ maps to~0 in $A_{\fm}$. Since the $gm_i$ span $V$, we find that $M_{\fm}=0$, as claimed.
\end{proof}

We now reach the main result:

\begin{proposition}
Let $M$ be an $A$-module. Then $\phi_M$ induces an isomorphism after localizing at $\fm$.
\end{proposition}

\begin{proof}
The assignment $M \mapsto \phi_M$ commutes with filtered colimits, and so it suffices to treat the case where $M$ is finitely generated. Let $R$ be the kernel of $\phi_M$, which is an $A$-submodule of $M$ by Lemma~\ref{lem:ker}, and let $N=M/\fm M \otimes \bC[K]$. By Lemma~\ref{lem:phi-locm}, the localization of $\phi_M$ at $\fm$ is a surjection. Since localization is exact, we have an exact sequence of algebraically $H$-equivariant $A_{\fm}$-modules
\begin{displaymath}
0 \to R_{\fm} \to M_{\fm} \to N_{\fm} \to 0.
\end{displaymath}
From Lemma~\ref{lem:eqfg}, we conclude that $R_{\fm}$ is equivariantly finitely generated as an $A_{\fm}$-module. Let $V \subset R$ be a finite length algebraic $\Sp$-representation generating $R_{\fm}$ as an $\vert A_{\fm} \vert$-module, and let $R_0$ be the $A$-submodule of $R$ generated by $V$. Note that $R_0$ is finitely generated as an $A$-module and $(R_0)_{\fm}=R_{\fm}$. Now, the mod $\fm$ reduction of the above exact sequence is exact, by the freeness of $N_{\fm}$, and the reduction of $M_{\fm} \to N_{\fm}$ is an isomorphism by Lemma~\ref{lem:phi-modm}. We conclude that $R/\fm R=R_0/\fm R_0=0$. Lemma~\ref{lem:nak} thus shows that $(R_0)_{\fm}=0$ and so $R_{\fm}=0$, and the proposition is proved.
\end{proof}

\section{The generic category} \label{s:gen}

\subsection{Statement of results} \label{ss:gen-1}

Let $\Mod_A^{\tors}$ be the category of torsion $A$-modules. We define the {\bf generic category} $\Mod_A^{\gen}$ to be the Serre quotient category $\Mod_A/\Mod_A^{\tors}$. We write $T \colon \Mod_A \to \Mod_A^{\gen}$ for the localization functor and let $S$ be its right adjoint (the section functor). The goal of this section is to understand the structure of the generic category and the behavior of $T$ and $S$. We achieve this by relating the generic category to $\Rep(H)$.

Let $M$ be an $A$-module. Define $\Phi(M)=M/\fm M$, where $\fm$ is the maximal ideal of $\vert A \vert$ considered in the previous section. Since $\fm$ is $H$-stable, it follows that $\Phi(M)$ carries a representation of $H$. It is easily seen to be algebraic: indeed, we can express $M$ as a quotient of $A \otimes V$, for some polynomial representation $V$, and then $\Phi(M)$ is a quotient of $\Phi(A \otimes V)=V$, and thus algebraic. We have thus defined a functor
\begin{displaymath}
\Phi \colon \Mod_A \to \Rep(H).
\end{displaymath}
Since $\Phi$ is cocontinuous, and the categories involved are Grothendieck, it has a right adjoint $\Psi$. In fact, it is not difficult to show that
\begin{displaymath}
\Psi(V) = \Hom_{\cU(\fh)}(\cU(\fsp), V)^{\rm alg}
\end{displaymath}
where $(-)^{\rm alg}$ denotes the maximal algebraic subrepresentation of an $\fsp$-module. Since we will not need this fact, we do not discuss it further.

The following is the main theorem of this section:

\begin{theorem} \label{thm:gen}
We have the following:
\begin{enumerate}
\item The functor $\Phi$ is exact.
\item The kernel of $\Phi$ is $\Mod_A^{\tors}$.
\item The counit $\Phi \Psi \to \id$ is an isomorphism.
\item The functor $\Phi$ induces an equivalence $\Mod_A^{\gen} \to \Rep(H)$.
\item The unit $V \otimes A \to \Psi(\Phi(V \otimes A))$ is an isomorphism, for any $V \in \Rep^{\pol}(\GL)$.
\end{enumerate}
\end{theorem}

We can use the theorem to transfer our understanding of $\Rep(H)$ to $\Mod_A^{\gen}$:

\begin{corollary} \label{cor:gen}
We have the following:
\begin{enumerate}
\item If $M$ is a finitely generated $A$-module then $T(M)$ has finite length.
\item Every finite length object of $\Mod_A^{\gen}$ has finite injective dimension.
\item The injectives of $\Mod_A^{\gen}$ are exactly the objects $T(V \otimes A)$ with $V \in \Rep^{\pol}(\GL)$.
\item The unit $V \otimes A \to S(T(V \otimes A))$ is an isomorphism, for any $V \in \Rep^{\pol}(\GL)$.
\end{enumerate}
\end{corollary}

Theorem~\ref{thm:gen} is analogous to \cite[Theorem~3.1]{sym2noeth} and \cite[Theorem~6.1]{periplectic}. Our proof of Theorem~\ref{thm:gen} simplifies the proofs in those papers. We believe the method here could be used in those papers as well, and would yield significant simplifications.

\subsection{Proof of the theorem}

We require several lemmas before proving the theorem. We let $F_{\lambda}$ be the $A$-module $\bV_{\lambda} \otimes A$, and we let $\cF$ be the class of $A$-modules that are (possibly infinite) direct sums of $F_{\lambda}$'s.

\begin{lemma} \label{lem:gen1}
Let $f \colon M \to N$ be a morphism of $A$-modules such that $\Phi(f)=0$. Then the localized morphism $f_{\fm} \colon M_{\fm} \to N_{\fm}$ vanishes.
\end{lemma}

\begin{proof}
By Theorem~\ref{thm:loc}, we have a commutative diagram
\begin{displaymath}
\xymatrix@C=2cm{
M_{\fm} \ar[r] \ar[d]_{f_{\fm}} & \Phi(M) \otimes A_{\fm} \ar[d]^{\Phi(f) \otimes 1} \\
N_{\fm} \ar[r] & \Phi(N) \otimes A_{\fm} }
\end{displaymath}
where the horizontal maps are isomorphisms. The result follows.
\end{proof}

\begin{lemma} \label{lem:gen2}
For $F \in \cF$ and any partition $\lambda$, the map
\begin{equation} \label{eq:gen2}
\Phi \colon \Hom_A(F_{\lambda}, F) \to \Hom_H(\Phi(F_{\lambda}), \Phi(F))
\end{equation}
is an isomorphism.
\end{lemma}

\begin{proof}
The functors $\Phi$, $\Hom_A(F_{\lambda}, -)$, and $\Hom_H(\Phi(F_{\lambda}), -)$ all commute with arbitrary direct sums, so it suffices to treat the case where $F=F_{\mu}$ for some $\mu$. If $f \colon F_{\lambda} \to F_{\mu}$ is a morphism such that $\Phi(f)=0$ then $f_{\fm}=0$ by Lemma~\ref{lem:gen1}. Since $F_{\lambda}$ and $F_{\mu}$ inject into their localizations at $\fm$, it follows that $f=0$. Thus the morphism \eqref{eq:gen2} is injective. Since the domain and target of \eqref{eq:gen2} have the same dimension by Proposition~\ref{prop:dimcount}, it is therefore an isomorphism.
\end{proof}

\begin{lemma} \label{lem:gen3}
Let $f \colon M \to N$ be a map of $A$-modules. Suppose that for all partitions $\lambda$ the induced map
\begin{displaymath}
f_* \colon \Hom_A(F_{\lambda}, M) \to \Hom_A(F_{\lambda}, N)
\end{displaymath}
is an isomorphism. Then $f$ is an isomorphism.
\end{lemma}

\begin{proof}
This simply follows from the fact that the $F_{\lambda}$'s generate $\Mod_A$. Here are some details. Suppose that $g \colon F_{\lambda} \to \ker(f)$ is some map. Then $fg=0$. Since $f_*$ is an isomorphism, it follows that $g=0$. Thus $\Hom_A(F_{\lambda}, \ker(f))=0$ for all $\lambda$. Since $\ker(f)$ is a quotient of a sum of $F_{\lambda}$'s, we see that $\ker(f)=0$. Thus $f$ is injective.

Now let $g \colon F_{\lambda} \to N$ be some morphism. Since $f_*$ is an isomorphism, we can write $g=fg'$ for some morphism $g' \colon F_{\lambda} \to M$. Thus $\im(g) \subset \im(f)$, and so the composition $F_{\lambda} \to N \to \coker(f)$ vanishes. Now, let $F \to N$ be a surjection with $F$ a sum of $F_{\lambda}$'s. Then the induced map $F \to \coker(f)$ is both zero and surjective. It follows that $\coker(f)=0$, and so $f$ is surjective.
\end{proof}

\begin{lemma} \label{lem:gen4}
For $F \in \cF$, the unit $\eta_F \colon F \to \Psi(\Phi(F))$ is an isomorphism.
\end{lemma}

\begin{proof}
Let $\lambda$ be a partition. We have a commutative diagram
\begin{displaymath}
\xymatrix{
\Hom_A(F_{\lambda}, F) \ar[rr]^{(\eta_F)_*} \ar[rd]_{\Phi} &&
\Hom_A(F_{\lambda}, \Psi(\Phi(F))) \ar[ld]^i \\
& \Hom(\Phi(F_{\lambda}), \Phi(F)) }
\end{displaymath}
where $i$ is the adjunction isomorphism. Since $\Phi$ is an isomorphism by Lemma~\ref{lem:gen2}, it follows that $(\eta_F)_*$ is an isomorphism. Thus $\eta_F$ is an isomorphism by Lemma~\ref{lem:gen3}.
\end{proof}

\begin{lemma} \label{lem:gen5}
Let $I$ be an injective object of $\Rep(H)$. Then the counit $\epsilon_I \colon \Phi(\Psi(I)) \to I$ is an isomorphism.
\end{lemma}

\begin{proof}
By the classification of injectives in $\Rep(H)$, we can write $I=\Phi(F)$ for some $F \in \cF$. Consider the diagram
\begin{displaymath}
\xymatrix@C=2cm{
\Phi(F) \ar[r]^-{\Phi(\eta_F)} & \Phi(\Psi(\Phi(F))) \ar[r]^-{\epsilon_{\Phi(F)}} \ar@{=}[d] & \Phi(F) \ar@{=}[d] \\
& \Phi(\Psi(I)) \ar[r]^-{\epsilon_I} & I }
\end{displaymath}
By basic properties of adjunction, the composition in the first line is the identity. By Lemma~\ref{lem:gen4}, the unit $\eta_F$ is an isomorphism. Thus $\Phi(\eta_F)$ is an isomorphism as well, and so $\epsilon_{\Phi(F)}$ is an isomorphism, and so $\epsilon_I$ is an isomorphism.
\end{proof}

\begin{proof}[Proof of Theorem~\ref{thm:gen}]
(a) Suppose that
\begin{displaymath}
0 \to M \to M' \to M'' \to 0
\end{displaymath}
is an exact sequence of $A$-modules. Since localization is exact, the sequence of $\vert A_{\fm} \vert$-modules
\begin{displaymath}
0 \to M_{\fm} \to M'_{\fm} \to M''_{\fm} \to 0
\end{displaymath}
is also exact. Since $M''_{\fm}$ is free as an $\vert A_{\fm} \vert$-module by Theorem~\ref{thm:loc}, the sequence remains exact after applying $- \otimes_{A_{\fm}} A_{\fm}/\fm$. We thus find that the sequence
\begin{displaymath}
0 \to \Phi(M) \to \Phi(M') \to \Phi(M'') \to 0
\end{displaymath}
is exact, which proves the statement.

(b) Let $M$ be an $A$-module such that $\Phi(M)=M/\fm M$ is zero. By Theorem~\ref{thm:loc}, we find that $M_{\fm}=0$. Thus every element of $M$ is annihilated by a non-zero element of $A$, and so $M$ is a torsion module.

(c) Let $V$ be an algebraic representation of $H$. Choose a co-presentation $0 \to V \to I \to J$ where $I$ and $J$ are injectives of $\Rep(H)$. Consider the diagram
\begin{displaymath}
\xymatrix{
0 \ar[r] & V \ar[r] & I \ar[r] & J \\
0 \ar[r] & \Phi(\Psi(V)) \ar[r] \ar[u]_{\epsilon_V} & \Phi(\Psi(I)) \ar[r] \ar[u]_{\epsilon_I} & \Phi(\Psi(J)) \ar[u]_{\epsilon_J}
}
\end{displaymath}
The bottom row is exact since the functor $\Phi \circ \Psi$ is left exact. The maps $\epsilon_I$ and $\epsilon_J$ are isomorphisms by Lemma~\ref{lem:gen5}, so $\epsilon_V$ is as well.

(d) This is a consequence of (a)--(c) and \cite[Prop.~III.5]{gabriel}.

(e) This was proved in Lemma~\ref{lem:gen4}.
\end{proof}

\section{Structure of $A$-modules} \label{s:Amod}

\subsection{The Artin--Rees lemma and consequences} \label{ss:artin-rees}

Let $I=A_+$ be the ideal of positive degree elements of $A$ and let $\cR=\bigoplus_{n \ge 0} I^n$ be the corresponding Rees algebra (also called the blow-up algebra). Then $\cR$ is naturally an algebra object in $\Rep(\Sp)$; in fact, it is a graded algebra, with $I^n$ having degree $n$.

\begin{proposition}
$\cR$ is noetherian as an (ungraded) algebra in $\Rep(\Sp)$.
\end{proposition}

\begin{proof}
We have a surjection $A \otimes \bV \to I$ of $A$-modules. We thus see that $\cR$ is a quotient of $A \otimes \Sym(\bV) = \Sym(\bV \oplus \bV)$ as an algebra. Now, $\Sym(\bV^{\oplus 2})$ is a noetherian algebra in $\Rep(\Sp)$, by the same reasoning used in the proof of Proposition~\ref{prop:Anoeth}; note that the restriction of $\Sym(\bV^{\oplus 2})$ to $\fS$ is the algebra $R_4$ of \S \ref{ss:repSym}(\ref{sym:noeth}). Since $\cR$ is a quotient of the noetherian algebra $\Sym(\bV^{\oplus 2})$, it too is noetherian.
\end{proof}

\begin{proposition}[Artin--Rees lemma] \label{prop:artin-rees}
Let $M \subset N$ be $A$-modules, with $N$ finitely generated. Then there exists an integer $k$ such that $M \cap I^n N = I^{n-k} (M \cap I^k N)$ holds for all $n \ge k$.
\end{proposition}

\begin{proof}
Define $\cN=\bigoplus_{n \ge 0} I^n N$. This is naturally a graded $\cR$-module, and is finitely generated. Let $\cM=\bigoplus_{n \ge 0} M \cap I^n N$; this a homogeneous $\cR$-submodule of $\cN$. Since $\cR$ is noetherian, it follows that $\cM$ is finitely generated. There is therefore some $k$ such that $\cM_k$ generates $\cM_n$ for all $n \ge k$. This yields the result.
\end{proof}

\begin{corollary}
Let $N$ be a finitely generated $A$-module. Then there exists an integer $n$ such that $I^n N$ is torsion-free.
\end{corollary}

\begin{proof}
Let $T$ be the torsion submodule of $N$; this is finitely generated by noetherianity. By Proposition~\ref{prop:artin-rees}, there is some $k$ such that $T \cap I^n N \subset I^{n-k} T$ for all $n \ge k$. Taking $n$ such that $I^{n-k} T=0$, we see that $T \cap I^n N=0$, and so $I^n N$ is torsion-free.
\end{proof}

\begin{corollary} \label{cor:inj}
Let $J$ be an injective object in the category $\Mod_A^{\tors}$. Then $J$ is injective in the category $\Mod_A$.
\end{corollary}

\begin{proof}
Let $N \subset M$ be finitely generated $A$-modules and let $f \colon N \to J$ be a morphism of $A$-modules. Let $K$ be the kernel of $f$; note that $N/K$ injects into $J$, and is thus torsion. Let $n$ be such that $I^n(M/K)$ is torsion free. The torsion submodule of $M/K$ thus injects into $M/(I^nM+K)$; in particular, $N/K$ injects into $M/(I^nM+K)$. Let $\ol{f} \colon N/K \to J$ be the morphism induced by $f$. Since $J$ is injective in $\Mod_A^{\tors}$, we can extend $\ol{f}$ to a morphism $\ol{g} \colon M/(I^nM+K) \to J$. Composing $\ol{g}$ with the quotient map $M \to M/(I^nM+K)$ yields a morphism $g \colon M \to J$ extending $f$. It follows that $J$ satisfies the necessary condition to be injective with respect to morphisms of finitely generated $A$-modules. Since $\Mod_A$ is locally noetherian, it follows from a version of Baer's criterion (see \cite[Proposition~A.14]{increp}) that $J$ is injective.
\end{proof}

\subsection{Saturation and local cohomology}

We now develop a theory of saturation and local cohomology for $A$-modules. We refer to \cite[\S 4]{symu1} for background. We note that the important property (Inj) of loc.\ cit.\ holds by Corollary~\ref{cor:inj}.

We have a left-exact functor $\Gamma \colon \Mod_A \to \Mod_A^\tors$ where $\Gamma(M)$ is the torsion submodule of $M$. Its derived functors $\rR^i \Gamma$ are the local cohomology functors. We also have a left exact functor $\Sigma \colon \Mod_A \to \Mod_A$ given by $\Sigma = S \circ T$ and called saturation. The following is \cite[Proposition 4.2]{symu1}. 

\begin{theorem} \label{thm:tri}
For any $M \in \rD^+(\Mod_A)$, we have a canonical exact triangle
\begin{displaymath}
\rR \Gamma(M) \to M \to \rR \Sigma(M) \to
\end{displaymath}
If $M \in \rD^b_{\fgen}(\Mod_A)$ then $\rR \Gamma(M)$ and $\rR \Sigma(M)$ are also in $\rD^b_{\fgen}(\Mod_A)$; in fact, $\rR \Gamma(M)$ is quasi-isomorphic to a finite length complex of finite length modules and $\rR \Sigma(M)$ is quasi-isomorphic to a finite length complex of modules of the form $V \otimes A$ with $V$ a finite length polynomial representation.
\end{theorem}

\begin{proof}
The existence of the triangle is \cite[Proposition 4.6]{symu1}. Suppose now that $M \in \rD^b_{\fgen}(\Mod_A)$. Then $T(M) \in \rD^b_{\fgen}(\Mod_A^{\gen})$. By Corollary~\ref{cor:gen}, we can therefore find a quasi-isomorphism $T(M) \to I^{\bullet}$ where $I^{\bullet}$ is a bounded complex whose terms have the form $T(V \otimes A)$ where $V$ is a finite length polynomial representation. Since $I^{\bullet}$ is a complex of injectives, we have $\rR \Sigma(M) = \rR S(T(M)) \cong S(I^{\bullet})$, which is a finite length complex whose terms have the form $V \otimes A$ where $V$ is a finite length polynomial representation. We thus see that $\rR \Sigma(M) \in \rD^b_{\fgen}(\Mod_A)$. From the triangle in the statement of the theorem, it now follows that $\rR \Gamma(M)$ belongs to $\rD^b_{\fgen}(\Mod_A)$. Since its cohomology groups are torsion, one can show that it is quasi-isomorphic to a finite length complex of finite length modules.
\end{proof}

\begin{corollary}
Let $M$ be a finitely generated $A$-module. Then we have a $4$-term exact sequence
\begin{displaymath}
0 \to \Gamma(M) \to M \to \Sigma(M) \to \rR^1 \Gamma(M) \to 0
\end{displaymath}
and isomorphisms $\rR^i \Sigma(M) \cong \rR^{i+1} \Gamma(M)$ for $i \ge 1$. The groups $\rR^i \Gamma(M)$ are finite length $A$-modules for all $i$ and vanish for $i \gg 0$.
\end{corollary}

\begin{corollary} \label{cor:der-gen}
The category $\rD^b_{\fgen}(\Mod_A)$ is generated (as a triangulated category) by the modules $\bV_{\lambda} \otimes A$ and the $A/A_+$-modules $\bV_{\lambda}$.
\end{corollary}

\subsection{Injective modules}

We have the following structural result for injective $A$-modules.

\begin{proposition} \label{prop:inj-decomp}
Let $I$ be a an injective $A$-module. Then $I$ decomposes as $I' \oplus I''$ where $I'$ is a torsion injective module and $I''$ is a torsion-free injective module.
\end{proposition}

\begin{proof}
This follows formally from property (Inj), i.e., Corollary~\ref{cor:inj}; see \cite[Proposition~4.3]{symu1} for details.
\end{proof}

The torsion-free injective modules are classified by the following result. We do not have a good understanding of the torsion injectives; see Example~\ref{ex:tors-inj} for one observation.

\begin{proposition} \label{prop:inj-tors-free}
For any $V \in \Rep^{\pol}(\GL)$ the $A$-module $V \otimes A$ is injective, and every torsion-free injective $A$-module is of this form. In particular, the indecomposable torsion-free injective $A$-modules are exactly the modules $\bV_{\lambda} \otimes A$.
\end{proposition}

\begin{proof}
By Corollary~\ref{cor:gen}(c), the object $T(V \otimes A)$ is injective in $\Mod_A^{\gen}$. Since $S$ takes injectives to injectives, it follows that $S(T(V \otimes A))$ is injective in $\Mod_A$. But by Corollary~\ref{cor:gen}(d), this is $V \otimes A$. Thus $V \otimes A$ is injective, and it is clearly torsion-free. If $I$ is a torsion-free injective then $T(I)$ is injective \cite[Proposition~4.3]{symu1}, and thus of the form $T(V \otimes A)$ by Corollary~\ref{cor:gen}(c). Since $I \cong S(T(I))$ \cite[Proposition~4.3]{symu1}, we see that $I \cong V \otimes A$.
\end{proof}

\subsection{The Grothendieck group}

For a locally noetherian abelian category $\cA$, we let $\rK(\cA)$ be the Grothendieck group of the category of finitely generated objects in $\cA$. We put $\rK(A)=\rK(\Mod_A)$ and $\rK(\Sp)=\rK(\Rep(\Sp))$. The tensor product on $\Rep(\Sp)$ gives $\rK(\Sp)$ the structure of a commutative ring. Similarly, the functor $\Rep(\Sp) \times \Mod_A \to \Mod_A$ given by $(V, M) \mapsto V \otimes M$ gives $\rK(A)$ the structure of a $\rK(\Sp)$-module. Let $\Lambda$ be the ring of symmetric functions. Then we have a natural ring isomorphism $\Lambda \to \rK(\Sp)$ taking $s_{\lambda}$ to $\bV_{\lambda}$ (that it is a ring homomorphism is clear since $\bV_\lambda$ are Schur functors applied to $\bV$; that it is an isomorphism follows from the fact that the change of basis between $\bV_\lambda$ and simple objects is upper unitriangular, for example by \S\ref{ss:repSp}\eqref{sp:simples}). We can thus regard $\rK(A)$ as a $\Lambda$-module. The following result gives its structure:

\begin{theorem}
The Grothendieck group $\rK(A)$ is a free module over $\Lambda$ of rank two. The classes $[\bC]$ and $[A]$ form a $\Lambda$-basis, where $\bC$ is regarded as an $A$-module via $\bC=A/A_+$.
\end{theorem}

\begin{proof}
By general properties of Serre quotients, we have a canonical exact sequence
\begin{displaymath}
\xymatrix{
\rK(\Mod_A^{\tors}) \ar[r]^-i & \rK(A) \ar[r]^-{\pi} & \rK(\Mod_A^{\gen}) \ar[r] & 0. }
\end{displaymath}
We also have a map
\begin{displaymath}
\gamma \colon \rK(A) \to \rK(\Mod_A^{\tors}), \qquad
\gamma([M]) = \sum_{i \ge 0} (-1)^i [\rR^i \Gamma(M)],
\end{displaymath}
which is well-defined by Theorem~\ref{thm:tri}. If $M$ is torsion then $T(M)=0$, and so $\rR \Sigma(M)=0$, and so $\rR \Gamma(M) \cong M$ by Theorem~\ref{thm:tri}. Thus $\gamma \circ i = \id$, and so $i$ is an injection. Since every finitely generated torsion module has a finite filtration such that $A_+$ acts by zero on the associated graded, it follows that $\rK(\Mod_A^{\tors})=\rK(\Sp)$ is a free $\Lambda$-module of rank~1, generated by $[\bC]$. By our analysis of $\Mod_A^{\gen} \cong \Rep(H)$, we know that the classes $[T(\bV_{\lambda} \otimes A)]$ form a $\bZ$-basis of $\rK(\Mod_A^{\gen})$, and so $\rK(\Mod_A^{\gen})$ is a free $\Lambda$-module of rank~1 with basis $[T(A)]$. Since $\pi([A])=[T(A)]$, we thus see that $[\bC]$ and $[A]$ form a $\Lambda$-basis for $\rK(A)$.
\end{proof}

\subsection{Torsion $A$-modules}

In \cite{symc1}, we saw that finite length $\GL$-equivariant $\Sym(\bV)$-modules enjoy nice homological properties: every such module has finite injective dimension and a finite length injective envelope. We now observe, by example, that these properties fail for $\Sp$-equivariant $\Sym(\bV$)-modules.

\begin{example} \label{ex:trivial-ext}
    Let $\bC = A/A_+$. We claim that $\ext^{2i}_A(\bC,\bC) = \bC$ for all $i \ge 0$ and is 0 in odd degrees. Consider the Koszul resolution $\bK_\bullet$ of $\bC=A/A_+$ given by $\bK_i = A \otimes \bigwedge^i \bV$. The terms of this complex are not projective. However, for any $V \in \Rep(\Sp)$ and $i>0$, we have
\begin{displaymath}
\Ext^i_A(A \otimes V, \bC) = \Ext^i_{\Sp}(V, \bC) = 0,
\end{displaymath}
where the first identification comes from adjunction and the second from the fact that $\bC$ is injective in $\Rep(\Sp)$ (\S \ref{ss:repSp}(\ref{sp:inj})). It follows that the terms of the Koszul complex are acyclic for the functor $\Ext^{\bullet}_A(-, \bC)$, and so the Koszul resolution can be used to compute these $\Ext$ groups. We find
\begin{displaymath}
\Ext^i_A(\bC, \bC) = \Hom_{\Sp}(\lw^i{\bV}, \bC).
\end{displaymath}
This vanishes if $i$ is odd (this parity argument follows from \S\ref{ss:repSp}(\ref{sp:tca})) and otherwise is 1-dimensional (this can be deduced from the branching rule from $\GL$ to $\Sp$ in \cite[7.9]{infrank}).
\end{example}

\begin{example}
  \addtocounter{equation}{-1}
  \begin{subequations}
    We have just seen that $\Ext^2_A(\bC, \bC)$ is one-dimensional. We now describe an explicit Yoneda 2-extension representing the non-zero class. To begin, we have an exact sequence
\begin{equation} \label{eq:yoneda-1}
0 \to \bV \to A/A_+^2 \to \bC \to 0
\end{equation}
Tensoring this sequence with $\bV$, we obtain an exact sequence
\begin{displaymath}
0 \to \bV^{\otimes 2} \to \bV \otimes A/A_+^2 \to \bV \to 0
\end{displaymath}
The symplectic form $\omega$ gives a map $\bV^{\otimes 2} \to \bC$. We can therefore push-out the above extension along this map to obtain an extension
\begin{displaymath}
0 \to \bC \to E \to \bV \to 0,
\end{displaymath}
where $E$ is a quotient of $\bV \otimes A/A_+^2$. We can now splice this extension with \eqref{eq:yoneda-1} to obtain a 2-extension
\begin{displaymath}
0 \to \bC \to E \to A/A_+^2 \to \bC \to 0.
\end{displaymath}
This represents the non-zero element of $\Ext^2_A(\bC, \bC)$.
\end{subequations}
\end{example}

\begin{example} \label{ex:tors-inj}
Let $v$ be an element of $\bV$. Then $v$ defines a linear functional $\omega(v, -)$ on $\bV$, and thus a derivation $\partial_v$ on $A$ (essentially a partial derivative). For $v,w \in \bV$ the operators $\partial_v$ and $\partial_w$ commute. We can therefore use them to define a new $A$-module structure on $A$, which we denote by $A^*$. Explicitly, if $x=v_1 \cdots v_n$ is a monomial in $A$ and $f \in A^*$, the product $xf$ is defined to be $\partial_{v_1} \cdots \partial_{v_n}(f)$. The $A/A_+$-module $\bC$ is an $A$-submodule of $A^*$, as constants are killed by all partial derivatives. If $f$ is any non-zero element of $A^*$ then we can find $v_1, \ldots, v_n \in \bV$ such that $\partial_{v_1} \cdots \partial_{v_n}(f)$ is a non-zero constant. We thus see that any non-zero $A$-submodule of $A^*$ must contain $\bC$. This shows that $A^*$ is an essential extension of $\bC$. It follows that the injective envelope of $\bC$ must contain $A^*$, and therefore does not have finite length.
\end{example}

\section{Applications to tca's} \label{s:Bnoeth}

\subsection{The noetherian property} \label{ss:noeth-prop}

Let $B=\Sym(\bV \oplus \lw^2{\bV})$, regarded as a tca. The goal of this section is to prove the following theorem:

\begin{theorem} \label{thm:Bnoeth}
The tca $B$ is noetherian.
\end{theorem}

Write $B=B_1 \otimes B_2$ where $B_1=\Sym(\bV$) and $B_2=\Sym(\lw^2{\bV})$. We say that a $B_2$-module is {\bf torsion} if every element has non-zero annihilator in $B_2$. We write $\Mod_{B_2}^{\tors}$ for the category of torsion modules, and $\Mod_{B_2}^{\gen}$ for the Serre quotient category $\Mod_{B_2}/\Mod_{B_2}^{\tors}$. The symplectic form $\omega$ on $\bV$ induces an $\Sp$-equivariant ring homomorphism $B_2 \to \bC$. Let $\Phi \colon \Mod_{B_2} \to \Rep(\Sp)$ be the functor $\Phi(M)=M \otimes_{B_2} \bC$. We show in \cite[Theorem~3.1]{sym2noeth} (see also \cite[\S 3.5]{sym2noeth}) that $\Phi$ is exact with kernel $\Mod_{B_2}^{\tors}$, and that the induced functor $\Mod_{B_2}^{\gen} \to \Rep(\Sp)$ is an equivalence. Let $\Psi$ be the right adjoint of $\Phi$.

\begin{lemma}
We have the following:
\begin{enumerate}
\item $\Phi$ induces a functor $\wt{\Phi} \colon \Mod_B \to \Mod_A$.
\item $\Psi$ induces a functor $\wt{\Psi} \colon \Mod_A \to \Mod_B$.
\item The functors $(\wt{\Phi}, \wt{\Psi})$ form an adjoint pair.
\item The derived functor of $\wt{\Psi}$ coincides with that of $\Psi$ on the bounded below derived category, i.e., the $B_2$-module underlying $\rR \wt{\Psi}(M)$ is $\rR \Psi(M)$.
\end{enumerate}
\end{lemma}

\begin{proof}
Let $K=\Frac(B_2)$. By a \emph{$K$-module} we mean a $K$-vector space $V$ equipped with a semilinear action of $\GL$ such that there exists a $\bC$-subspace $W$ of $V$ that forms a polynomial representation of $\GL$ and spans $V$ over $K$. Let $\Mod_K$ be the category of $K$-modules. There is a natural functor $\Phi' \colon \Mod_{B_2} \to \Mod_K$ given by $\Phi'(M) = K \otimes_{B_2} M$. There is also a natural functor $\Psi' \colon \Mod_K \to \Mod_{B_2}$ given by $\Psi'(V)=V^{\pol}$, the maximal polynomial subrepresentation of $V$. In \cite[\S 2.4]{sym2noeth}, we show that $(\Phi', \Psi')$ are an adjoint pair, and that $\Phi'$ induces an equivalence $\Mod_{B_2}^{\gen} \to \Mod_K$. Thus $\Mod_K$ is equivalent to $\Rep(\Sp)$, and under this equivalence $\Phi'$ and $\Psi'$ correspond to $\Phi$ and $\Psi$. We can therefore prove the proposition relative to $\Phi'$ and $\Psi'$. Note $A \in \Rep(\Sp)$ corresponds to the algebra object $A'=K \otimes \Sym(\bV)$ of $\Mod_K$. Thus $\Mod_A$ corresponds to the category of $\Mod_{A'}$ of $A'$-modules in $\Sym(K)$. We note that $B$ is naturally a $\GL$-stable subalgebra of $A'$.

(a) If $M$ is a $B$-module then $A' \otimes_B M$ is an $A'$-module. This construction defines a functor $\wt{\Phi}' \colon \Mod_B \to \Mod_{A'}$.

(b) Let $M$ be a $A'$-module in $\Mod_K$. Then $M$ is a $\vert B \vert$-module (by restricting scalars), and we claim that $M^{\pol}$ is a $\vert B \vert$-submodule of $M$. Indeed, $B \otimes M^{\pol}$ is a polynomial representation of $\GL$, and so its image under the $\GL$-equivariant map $B \otimes M \to M$ must have polynomial image, and therefore be contained in $M^{\pol}$. Thus $\Psi'(M)=M^{\pol}$ is an object of $\Mod_B$, and so $\Psi'$ induces a functor $\wt{\Psi}' \colon \Mod_{A'} \to \Mod_B$.

(c) Let $M$ be an $A'$-module and let $N$ be a $B$-module. To prove the statement, it suffices to show that the unit $N \to \Psi'(\Phi'(N))$ is a map of $B$-modules and the counit $\Phi'(\Psi'(M)) \to M$ is a map of $A$-modules. This is clear from the definitions.

(d) We now switch back to $\Rep(\Sp)$ and $A$-modules. Let $M \in \rD^+(\Mod_A)$. Let $M \to I$ be a quasi-isomorphism with $I$ a bounded below complex of injective $A$-modules. Then $\wt{\Psi}(I)$ computes $\rR \wt{\Psi}(M)$. Since every injective $A$-module is injective as an $\Sp$-module by Proposition~\ref{prop:forget-injective}, it follows that $\Psi(I)$ computes $\rR \Psi(M)$. Since $\Psi(I)$ is the $B_2$-module underlying $\wt{\Psi}(I)$, the result follows.
\end{proof}

In what follows, we just write $\Phi$ and $\Psi$ for the functors on $\Mod_B$ and $\Mod_A$. We say that a $B$-module $M$ satisfies property (FT) if $\Tor^B_i(M, \bC)$ is finite length for all $i$.

\begin{lemma} \label{lem:Bnoeth-7}
Suppose that
\begin{displaymath}
0 \to M_1 \to M_2 \to M_3 \to 0
\end{displaymath}
is an exact sequence of $B$-modules. If two of the modules satisfy (FT) then so does the third.
\end{lemma}

\begin{proof}
This follows immediately from the long exact sequence in $\Tor$.
\end{proof}

\begin{lemma} \label{lem:Bnoeth-2}
Suppose $M$ is a finitely generated $B$-module that is $B_2$-torsion. Then:
\begin{enumerate}
\item $V \otimes M$ is a noetherian $B$-module, for any finite length polynomial representation $V$;
\item $M$ satisfies (FT).
\end{enumerate}
\end{lemma}

\begin{proof}
Let $M'$ be a finitely generated $B_2$-submodule of $M$ that generates $M$ as a $B$-module. Since $M'$ is a finitely generated torsion $B_2$-module, it is annihilated by a non-zero $\GL$-stable ideal $\fa$ of $B_2$ (see \cite[Corollary~2.3]{sym2noeth}). Since $M'$ generates $M$, it follows that $M$ is also annihilated by $\fa$. We thus see that we can regard $M$ as a module over the tca $C=B/\fa B=B_1 \otimes (B_2/\fa)$. Since $B_1$ and $B_2/\fa$ are essentially bounded in the sense of \cite[\S 2.3]{sym2noeth} (see \cite[Corollary~4.2]{sym2noeth}), so is $C$ by the Littlewood--Richardson rule. Thus $C$ is noetherian by \cite[Proposition~2.4]{sym2noeth}. Since $M$ is a finitely generated module over a noetherian tca, it is noetherian. Since $V \otimes M$ is also finitely generated and $B_2$-torsion, the preceding reasoning shows that it too is noetherian. This proves (a).

Now, $\Tor^B_{\bullet}(M, \bC)$ is computed by the Koszul complex $\lw^{\bullet}(\bV \oplus \lw^2{\bV}) \otimes M$. By (a), the terms of this complex are noetherian. Thus the homology groups are finitely generated $B$-modules. Since they are also annihilated by $B_+$, they have finite length, and (b) follows.
\end{proof}

Let $\cS$ be the class of objects $M$ in $\Mod_A$ such that $(\rR^i \Psi)(M)$  satisfies (FT) for all $i \ge 0$. Note in particular that if $M \in \cS$ then $\Psi(M)$ is finitely generated. Also note that, by general properties of Serre quotients and section functors, $(\rR^i \Psi)(M)$ is $B_2$-torsion for any $M$ and any $i>0$.

\begin{lemma} \label{lem:Bnoeth-3}
Suppose that
\begin{displaymath}
0 \to M_1 \to M_2 \to M_3 \to 0
\end{displaymath}
is an exact sequence in $\Mod_A$. If two of the modules belong to $\cS$ then so does the third.
\end{lemma}

\begin{proof}
Suppose that $M_1$ and $M_2$ belong to $\cS$. Consider the exact sequence
\begin{displaymath}
0 \to \Psi(M_1) \to \Psi(M_2) \to \Psi(M_3) \to N \to 0
\end{displaymath}
where $N$ is the image of $\Psi(M_3)$ in $\rR^1 \Psi(M_1)$. Since $\rR^1 \Psi(M_1)$ is $B_2$-torsion and finitely generated (as $M_1 \in \cS$), it is noetherian by Lemma~\ref{lem:Bnoeth-2}(a). Thus $N$ is finitely generated. Since $N$ is also $B_2$-torsion, as it is a submodule of $\rR^1 \Psi(M_1)$, it satisfies (FT) by Lemma~\ref{lem:Bnoeth-2}(b). From the above exact sequence and Lemma~\ref{lem:Bnoeth-7}, it follows that $\Psi(M_3)$ satisfies (FT). Now let $i \ge 1$, and consider the exact sequence
\begin{displaymath}
\rR^i \Psi(M_2) \to \rR^i \Psi(M_3) \to \rR^{i+1} \Psi(M_1).
\end{displaymath}
The outside groups are finitely generated by assumption and $B_2$-torsion; thus they are noetherian by Lemma~\ref{lem:Bnoeth-2}(a). It follows that $\rR^i \Psi(M_3)$ is finitely generated. Since it is also $B_2$-torsion, it satisfies (FT) by Lemma~\ref{lem:Bnoeth-2}(b). Thus $M_3$ belongs to $\cS$. The other cases are similar, and left to the reader.
\end{proof}

\begin{lemma} \label{lem:Bnoeth-6}
Let $M$ be a $B$-module that is projective as a $B_2$-module. Then the natural map $M \to \Psi(\Phi(M))$ is an isomorphism of $B$-modules, and $\rR^i \Psi(\Phi(M))=0$ for $i>0$.
\end{lemma}

\begin{proof}
The fact that $M \to \Psi(\Phi(M))$ is an isomorphism follows from the fact that projective $B_2$-modules are saturated \cite[Proposition~2.8]{sym2noeth}. The fact that $\rR^i \Psi(\Phi(M))=0$ for $i>0$ follows from the fact that if $M=V \otimes B_2$ is a projective $B_2$-module (with $V$ a polynomial representation of $\GL$) then $\Phi(M)=V$ is injective in $\Rep(\Sp)$ \cite[4.2.9]{infrank}.
\end{proof}

\begin{lemma} \label{lem:Bnoeth-4}
The $A$-modules $\bV_{\lambda}$ and $\bV_{\lambda} \otimes A$ belong to $\cS$.
\end{lemma}

\begin{proof}
We have $\bV_{\lambda}=\Phi(\bV_{\lambda} \otimes B_2)$; here $B_2$ is a $B$-module via $B_2=B/(B_1)_+$. Since $\bV_{\lambda} \otimes B_2$ is $B_2$-projective, we have $\Psi(\bV_{\lambda})=\bV_{\lambda} \otimes B_2$ and $\rR^i \Psi(\bV_{\lambda})=0$ for $i>0$ by Lemma~\ref{lem:Bnoeth-6}. Now, we have
\begin{displaymath}
\Tor^B_{\bullet}(B_2, \bC)=\Tor^{B_1}_{\bullet}(\bC, \bC)=\lw^{\bullet}(\bV),
\end{displaymath}
and so $B_2$ satisfies (FT) as a $B$-module; the same reasoning applies to $\bV_{\lambda} \otimes B_2$. We thus see that $\Psi(\bV_{\lambda})$ satisfies (FT), and so $\bV_{\lambda}$ belongs to $\cS$.

We have $\bV_{\lambda} \otimes A = \Phi(\bV_{\lambda} \otimes B)$. Since $\bV_{\lambda} \otimes B$ is $B_2$-projective, we have $\Psi(\bV_{\lambda} \otimes A) = \bV_{\lambda} \otimes B$ and $\rR^i \Psi(\bV_{\lambda} \otimes A)=0$ for $i>0$ by Lemma~\ref{lem:Bnoeth-6}. Since $\bV_{\lambda} \otimes B$ is finitely generated and projective, it satisfies (FT). Thus $\bV_{\lambda} \otimes A \in \cS$.
\end{proof}

\begin{lemma} \label{lem:Bnoeth-5}
Every finitely generated $A$-module belongs to $\cS$.
\end{lemma}

\begin{proof}
This follows from Lemmas~\ref{lem:Bnoeth-3} and~\ref{lem:Bnoeth-4} since the objects $\bV_{\lambda}$ and $\bV_{\lambda} \otimes A$ generate $\rD^b_{\fgen}(\Mod_A)$ (Corollary~\ref{cor:der-gen}). A more detailed argument follows.

Suppose $M$ is a finitely generated $A$-module with $A_+M=0$. Let $M \to I^{\bullet}$ be a finite length resolution in $\Rep(\Sp)$ where each $I^k$ is a finite length injective. Regard each $I^k$ as an $A/A_+$-module. Since each $I^k$ belongs to $\cS$ by Lemma~\ref{lem:Bnoeth-4}, we conclude that $M$ belongs to $\cS$ by Lemma~\ref{lem:Bnoeth-3}.

Now let $M$ be an arbitrary finite length $B$-module. We show that $M \in \cS$ by induction on the injective dimension of $T(M) \in \Mod_A^{\gen}$. Assume $M$ is non-zero. We can then find an exact sequence
\begin{displaymath}
0 \to T(M) \to I \to N \to 0
\end{displaymath}
in $\Mod_A^{\gen}$, where $I$ is a finite length injective and $N$ has smaller injective dimension than $T(M)$ (using the convention that the zero module has injective dimension $-\infty$). Applying $S$, we obtain an exact sequence
\begin{displaymath}
0 \to \Sigma(M) \to S(I) \to N' \to 0,
\end{displaymath}
where $N'$ is the image of $S(I)$ in $S(N)$. Since $T(N')=N$ has smaller injective dimension than $T(M)$, it follows that $N'$ belongs to $\cS$ by the inductive hypothesis. Since $S(I)$ is a finitely generated torsion-free injective $A$-module, it belongs to $\cS$ by Lemma~\ref{lem:Bnoeth-4}. Thus $\Sigma(M)$ belongs to $\cS$ by Lemma~\ref{lem:Bnoeth-3}. Now, we have an exact sequence
\begin{displaymath}
0 \to \Gamma(M) \to M \to \Sigma(M) \to \rR^1 \Gamma(M) \to 0.
\end{displaymath}
Since $\Gamma(M)$ and $\rR^1 \Gamma(M)$ are finite length $A$-modules, they belong to $\cS$ by the previous paragraph and Lemma~\ref{lem:Bnoeth-3}. Thus $M$ belongs to $\cS$ by Lemma~\ref{lem:Bnoeth-3}. This completes the proof.
\end{proof}

\begin{proof}[Proof of Theorem~\ref{thm:Bnoeth}]
It suffices to show that $M=\bV_{\lambda} \otimes B$ is a noetherian $B$-module for all $\lambda$. Let $N$ be a submodule. Then $\Phi(N)$ is an $A$-submodule of $\Phi(M)=\bV_{\lambda} \otimes A$. Since $\Phi(M)$ is finitely generated, it follows that $\Phi(N)$ is as well. Thus $\Phi(N)$ belongs to $\cS$ by Lemma~\ref{lem:Bnoeth-5}, and so $\Psi(\Phi(N))$ satisfies (FT). Since $M$ is $B_2$-projective, it has no $B_2$-torsion, and so the same is true for $N$. It follows that we have an exact sequence
\begin{displaymath}
0 \to N \to \Psi(\Phi(N)) \to T \to 0,
\end{displaymath}
where $T$ is $B_2$-torsion. Since $\Psi(\Phi(N))$ is finitely generated, so is $T$, and so $T$ satisfies (FT) by Lemma~\ref{lem:Bnoeth-2}. Thus $N$ satisfies (FT), and is therefore finitely generated. Thus $M$ is noetherian.
\end{proof}

\subsection{Some additional results}

We keep the notation $B=\Sym(\bV \oplus \lw^2{\bV})$. Let $\Mod_B^0$ be the full subcategory of $\Mod_B$ spanned by modules supported at~0 (i.e., every element is annihilated by a power of $B_+$), and let $\Mod_B^{\gen}$ be the generic category for $B$ (i.e., the Serre quotient of $\Mod_B$ by the torsion subcategory).

\begin{proposition}
We have an equivalence $\Mod_B^{\gen} \cong \Mod_B^0$.
\end{proposition}

\begin{proof}
The category $\Mod_B^{\gen}$ can be obtained by first forming the quotient of $\Mod_B$ by the subcategory of modules that are $B_2$-torsion, and then forming the quotient of the result by the subcategory of torsion modules. The first quotient gives $\Mod_A$ by the discussion in \S\ref{ss:noeth-prop}, so this gives an equivalence $\Mod_B^{\gen} \cong \Mod_A^{\gen}$. We have seen that $\Mod_A^{\gen}$ is equivalent to $\Rep(H)$ (Theorem~\ref{thm:gen}(d)), which is equivalent to $\Mod_B^0$ (Theorem~\ref{thm:Htca}).
\end{proof}

\begin{proposition}
Every projective $B$-module is injective.
\end{proposition}

\begin{proof}
Let $V$ be a polynomial representation. We must show that $V \otimes B$ is injective in $\Mod_B$ (every projective $B$-module is of this form). We know that $V \otimes A$ is injective in $\Mod_A$ (Proposition~\ref{prop:inj-tors-free}), and so $\Psi(V \otimes A)$ is injective in $\Mod_B$, since $\Psi$ is right adjoint to the exact functor $\Phi$. Since $V \otimes A = \Phi(V \otimes B)$, we have $\Psi(V \otimes A)=\Psi(\Phi(V \otimes B))=V \otimes B$ (Lemma~\ref{lem:Bnoeth-6}), and the result follows.
\end{proof}

\begin{remark}
Analogs of these propositions are known for $\Sym(\lw^2{\bV})$ and $\Sym(\Sym^2{\bV})$ \cite{sym2noeth} and $\Sym(\bV^{\oplus n})$ \cite{symc1,symu1}.
\end{remark}

\section{Further remarks} \label{s:further}

\subsection{A twisted Lie algebra} \label{ss:lie}

Let $\fg=\bV \oplus \lw^2{\bV}$. We give $\fg$ the structure of a Lie algebra as follows: (a) for $v,w \in \bV$, we put $[v,w]=v \wedge w \in \lw^2{\bV}$; and (b) all elements of $\lw^2{\bV}$ are central. Since the Lie bracket is $\GL$-equivariant, it follows that $\fg$ is a Lie algebra object in the category $\Rep^{\pol}(\GL)$. We write $\Mod_{\fg}$ for the category of $\fg$-modules internal to $\Rep^{\pol}(\GL)$.

\begin{theorem}
The category $\Mod_{\fg}$ is locally noetherian.
\end{theorem}

\begin{proof}
Let $F_n \cU(\fg)$ be the image of the natural map $\bigoplus_{i=0}^n \fg^{\otimes i} \to \cU(\fg)$. This defines a $\GL$-stable filtration $F_0{\cU(\fg)} \subset F_1{\cU(\fg)} \subset \cdots$ on $\cU(\fg)$ for which the associated graded is isomorphic to $B$. Now let $M=\bV_{\lambda} \otimes \cU(\fg)$, and define $F_i M=\bV_{\lambda} \otimes F_i\cU(\fg)$. Then $\gr(M) \cong \bV_{\lambda} \otimes B$. Let $N$ be a $\cU(\fg)$-submodule of $M$. Then $N$ inherits a filtration via $F_i N = N \cap F_iM$, and the natural map $\gr(N) \to \gr(M)$ is an injection of $B$-modules. Since $\gr(M)$ is a noetherian $B$-module (Theorem~\ref{thm:tca}), it follows that $\gr(N)$ is finitely generated as a $B$-module. Let $k$ be such that $\gr(N)$ is generated by $\gr_0(N), \ldots, \gr_k(N)$. Then a standard argument shows that $N$ is generated as a $\cU(\fg)$-module by $F_kN$, which is a finite length polynomial representation. Thus $N$ is finitely generated, and so $M$ is noetherian as a $\cU(\fg)$-module. Since every finitely generated $\fg$-module is a quotient of a finite direct sum of modules of the form $\bV_\lambda \otimes \cU(\fg)$, we see that any such module is noetherian.
\end{proof}

\subsection{Complementary examples}

In this paper, we have studied $\Sp$-equivariant modules over $\Sym(\bV)$. We can also consider the parallel situation of $\bO$-equivariant modules. To be somewhat precise, let $\beta$ be the symmetric bilinear form on $\bV$ defined by $\beta(e_i, f_j)=\delta_{i,j}$, and $\beta(e_i,e_j)=\beta(f_i,f_j)=0$, and let $\bO \subset \GL$ be the subgroup preserving $\beta$. We then have a category $\Rep(\bO)$ of algebraic representations of $\bO$, as studied in \cite{infrank}. Let $A'$ be the algebra $\Sym(\bV)$ in the category $\Rep(\bO)$. We can then consider $A'$-modules in the category $\Rep(\bO)$. Much of the above discussion extends to this setting. We define a linear form $\xi \colon \bV \to \bC$ by $\xi(e_i) = \xi(f_i) = 1$ for all $i$ and let $H'$ be the stabilizer of $\xi$ in $\bO$. Then one can show that the category of generic $A'$-modules is equivalent to the category of algebraic representations of $H'$ and use this to prove that the tca $\Sym(\bV \oplus \Sym^2 \bV)$ is noetherian.

There is one difference between $A$ and $A'$ to note. The graded pieces of $A$ are irreducible as $\Sp$-representations, but the graded pieces of $A'$ are not irreducible as $\bO$-representations. However, one can show that any non-zero $\bO$-stable ideal of $A'$ contains a power of $A'_+$, and so this difference does not affect too much.

There are, in fact, two additional examples one can consider. In \cite{infrank}, we also study algebraic representations of $\Sp$ and $\bO$ on pro-finite vector spaces. Let $\hat{\bV}$ be the dual space of $\bV$, which we can identify with an infinite product of $\bC$'s. The groups $\Sp$ and $\bO$ act on $\hat{\bV}$, and we can consider representations appearing in tensor powers (using a completed tensor product). We denote these categories by $\wh{\Rep}(\Sp)$ and $\wh{\Rep}(\bO)$. In fact, $\wh{\Rep}(\Sp)$ is equivalent to the opposite category of $\Rep(\Sp)$, and $\wh{\Rep}(\bO)$ is equivalent to the opposite category of $\Rep(\bO)$. We have algebra objects $\Sym(\hat{\bV})$ in both $\wh{\Rep}(\Sp)$ and $\wh{\Rep}(\bO)$, and one can consider their module categories. We have not carefully considered these examples, but expect that the methods of this paper could be used to study them. There is one new phenomenon worth pointing out: the representation $\Sym^2(\hat{\bV})$ of $\bO$ has an invariant element, and this generates a principal ideal of $\Sym(\hat{\bV})$ that does not contain any power of the maximal ideal.

One can also consider a variant of the Lie algebra from \S \ref{ss:lie}, as follows. Define the free 2-step nilpotent twisted Lie superalgebra $\fg' = \bV \oplus \Sym^2(\bV)$ in the category $\Rep^\pol(\GL)$. This is transpose dual to the twisted Lie algebra $\fg$, and hence we immediately deduce that $\Mod_{\fg'}$ is locally noetherian. (Recall that the transpose functor is the auto-equivalence of $\Rep^{\pol}(\GL)$ that takes $\bV_{\lambda}$ to $\bV_{\lambda^{\dag}}$, where $\lambda^{\dag}$ denotes the transposed partition. It is a non-symmetric monoidal equivalence. See \cite[\S 7.4]{expos}). Similarly, the twisted skew-commutative algebra $\lw(\bV) \otimes \Sym(\Sym^2 \bV)$ is transpose dual to the tca $B$ and its category of modules is also locally noetherian, while $\lw(\bV) \otimes \Sym(\lw^2 \bV)$ is transpose dual to the tca $\Sym(\bV \oplus \Sym^2 \bV)$ so its category of modules is locally noetherian too.

\subsection{Koszul duality}

In \cite{symc1}, we studied Koszul duality for the algebra object $\Sym(\bV)$ in $\Rep^{\pol}(\GL)$. The standard Koszul duality between $\Sym$ and $\lw$ induces a derived equivalence between $\Sym(\bV)$-modules in $\Rep^{\pol}(\GL)$ and $\lw(\bV)$-comodules in $\Rep^{\pol}(\GL)$. Composing this with the duality functor on $\Rep^{\pol}(\GL)$ and the transpose functor converts $\lw(\bV)$-comodules back to $\Sym(\bV)$-modules. In this way, one gets a contravariant derived auto-equivalence of the category of $\Sym(\bV)$-modules. We showed that this functor preserves the bounded finitely generated derived category.

One can carry out an analogous process in the present setting. Consider the algebra $\Sym(\bV)$ in $\Rep(\Sp)$. Koszul duality gives a derived equivalence between graded $\Sym(\bV)$-modules in $\Rep(\Sp)$ and graded $\lw(\bV)$-comodules in $\Rep(\Sp)$. Duality converts $\lw(\bV)$-comodules in $\Rep(\Sp)$ to $\lw(\hat{\bV})$-modules in $\wh{\Rep}(\Sp)$. Transpose now converts $\lw(\hat{\bV})$-modules in $\wh{\Rep}(\Sp)$ to $\Sym(\hat{\bV})$-modules in $\wh{\Rep}(\bO)$. We thus obtain a derived equivalence:
\begin{displaymath}
\{ \text{graded $\Sym(\bV)$-modules in $\Rep(\Sp)$} \} \leftrightarrow
\{ \text{graded $\Sym(\hat{\bV})$-modules in $\wh{\Rep}(\bO)$} \}.
\end{displaymath}
Similarly, we obtain a derived equivalence
\begin{displaymath}
\{ \text{graded $\Sym(\bV)$-modules in $\Rep(\bO)$} \} \leftrightarrow
\{ \text{graded $\Sym(\hat{\bV})$-modules in $\wh{\Rep}(\Sp)$} \}.
\end{displaymath}
(We note that objects of $\Rep^{\pol}(\GL)$ are canonically graded, which is why we did not need to include a grading in that setting.)


\begin{thebibliography}{Stacks}

\bibitem[AH]{AschenbrennerHillar} Matthias Aschenbrenner, Christopher J.~Hillar. Finite generation of symmetric ideals. \textit{Trans.\ Amer.\ Math.\ Soc.}, {\bf 359} (2007), 5171--5192; erratum, ibid.\ {\bf 361} (2009), 5627--5627. \arxiv{math/0411514v3}

\bibitem[Co]{cohen} D.~E. Cohen. On the laws of a metabelian variety. {\it J. Algebra} {\bf 5} (1967), 267--273.

\bibitem[Co2]{cohen2} D.~E.~Cohen. Closure relations, Buchberger's algorithm, and polynomials in infinitely many variables. In {\it Computation theory and logic}, volume 270 of {\it Lect. Notes Comput. Sci.}, pp.\ 78--87, 1987.

\bibitem[CEF]{fimodule}
Thomas Church, Jordan S. Ellenberg, Benson Farb. FI-modules and stability for representations of symmetric groups. {\it Duke Math. J.} {\bf 164} (2015), no.~9, 1833--1910. \arxiv{1204.4533v4}

\bibitem[DPS]{koszulcategory} Elizabeth Dan-Cohen, Ivan Penkov, Vera Serganova. A Koszul category of representations of finitary Lie algebras. \textit{Adv.\ Math.} {\bf 289} (2016), pp.~250--278. \arxiv{1105.3407v2}

\bibitem[HS]{HillarSullivant} Christopher J.~Hillar, Seth Sullivant. Finite Gr\"obner bases in infinite dimensional polynomial rings and applications. \textit{Adv.\ Math.} {\bf 229} (2012), no.~1, 1--25. \arxiv{0908.1777v2}

\bibitem[Gab]{gabriel} Pierre Gabriel. Des cat\'egories ab\'eliennes. {\it Bull. Soc. Math. France} {\bf 90} (1962), 323--448.

\bibitem[GS]{increp} Sema G\"urt\"urk\"un, Andrew Snowden. The representation theory of the increasing monoid. \arxiv{1812.10242v1}

\bibitem[NR]{nagel-romer} Uwe Nagel, Tim R\"omer. $\FI$- and $\OI$-modules with varying coefficients. \textit{J.\ Algebra} {\bf 535} (2019) 286--322. \arxiv{1710.09247v2}

\bibitem[NSS1]{sym2noeth} Rohit Nagpal, Steven V Sam, Andrew Snowden. Noetherianity of some degree two twisted commutative algebras. {\it Selecta Math. (N.S.)} {\bf 22} (2016), no.~2, 913--937. \arxiv{1501.06925v2}

\bibitem[NSS2]{periplectic} Rohit Nagpal, Steven V Sam, Andrew Snowden. Noetherianity of some degree two twisted skew-commutative algebras. {\it Selecta Math. (N.S.)} {\bf 25} (2019), no.~1. \arxiv{1610.01078v1}

\bibitem[Ols]{olshanskii} G.~I. Ol$'$shanski{\u \i}. Representations of infinite-dimensional classical groups, limits of enveloping algebras, and Yangians. {\it Topics in representation theory}, 1--66, Adv.\ Soviet Math., 2, Amer.\ Math.\ Soc., Providence, RI, 1991.

\bibitem[PSe]{penkovserganova} Ivan Penkov, Vera Serganova. Categories of integrable $sl(\infty)$-, $o(\infty)$-, $sp(\infty)$-modules. {\it Representation Theory and Mathematical Physics}, Contemp. Math. {\bf 557}, AMS 2011, pp. 335--357. \arxiv{1006.2749v1}

\bibitem[PSt]{penkovstyrkas} Ivan Penkov, Konstantin Styrkas. Tensor representations of classical locally finite Lie algebras. {\it Developments and trends in infinite-dimensional Lie theory}, Progr.\ Math.\ {\bf 288}, Birkh\"auser Boston, Inc., Boston, MA, 2011, pp. 127--150. \arxiv{0709.1525v1}

\bibitem[Sn]{delta-mod} Andrew Snowden. Syzygies of Segre embeddings and $\Delta$-modules. {\it Duke Math.\ J.} {\bf 162} (2013), no.~2, 225--277. \arxiv{1006.5248v4}

\bibitem[Stacks]{stacks}
Stacks Project, \url{http://stacks.math.columbia.edu}, 2020.

\bibitem[SS1]{symc1} Steven~V Sam, Andrew Snowden. GL-equivariant modules over polynomial rings in infinitely many variables. {\it Trans. Amer. Math. Soc.} {\bf 368} (2016), 1097--1158. \arxiv{1206.2233v3}

\bibitem[SS2]{expos}
Steven~V Sam, Andrew Snowden, Introduction to twisted commutative algebras, \arxiv{1209.5122v1}.

\bibitem[SS3]{infrank} Steven~V Sam, Andrew Snowden. Stability patterns in representation theory. {\it Forum Math. Sigma} {\bf 3} (2015), e11, 108 pp. \arxiv{1302.5859v2}

\bibitem[SS4]{spincat} Steven~V Sam, Andrew Snowden. Infinite rank spinor and oscillator representations. \textit{J.\ Comb.\ Algebra} (2017), no. 2, 145--183. \arxiv{1604.06368v2}

\bibitem[SS5]{symu1} Steven~V Sam, Andrew Snowden. GL-equivariant modules over polynomial rings in infinitely many variables. II. {\it Forum Math. Sigma} {\bf 7} (2019), e5, 71pp. \arxiv{1703.04516v2}

\bibitem[SS6]{brauercat} Steven~V Sam, Andrew Snowden. The representation theory of Brauer categories I: triangular categories. \arxiv{2006.04328v1}

\bibitem[SSW]{littlewood} Steven~V Sam, Andrew Snowden, Jerzy Weyman. Homology of Littlewood complexes. {\it Selecta Math. (N.S.)} {\bf 19} (2013), no.~3, 655--698. \arxiv{1209.3509v2}

\end{thebibliography}
\end{document}